\documentclass[11pt]{amsart}
\usepackage[utf8]{inputenc}

\usepackage{mathtools}
\usepackage{amssymb}
\usepackage{tikz}
\usepackage{tikz-cd}
\usetikzlibrary{decorations.pathmorphing}
\usepackage{pgfplots}
\pgfplotsset{compat=1.11}
\usepgfplotslibrary{fillbetween}
\usepackage{xcolor}
\usepackage{hyperref}
\usepackage{url}
\theoremstyle{plain}
\newtheorem{theorem}{Theorem}[section]
\newtheorem*{theorem*}{Theorem}
\newtheorem{proposition}[theorem]{Proposition}
\newtheorem{lemma}[theorem]{Lemma}
\newtheorem{corollary}[theorem]{Corollary}
\theoremstyle{remark}
\newtheorem{remark}[theorem]{Remark}

\newtheorem{definition}[theorem]{Definition}
\newtheorem{conjecture}[theorem]{Conjecture}

\newcommand{\CC}{\mathbb{C}}
\newcommand{\RR}{\mathbb{R}}
\newcommand{\PP}{\mathbb{P}}

\newcommand{\ZZ}{\mathbb{Z}}

\newcommand{\DD}{\mathcal{D}}
\newcommand{\EE}{\mathcal{E}}

\newcommand{\HH}{\mathcal{H}}
\newcommand{\KK}{\mathcal{K}}
\newcommand{\LL}{\mathcal{L}}
\newcommand{\cM}{\mathcal{M}}
\newcommand{\OO}{\mathcal{O}}
\newcommand{\Pm}{\mathcal{P}}

\newcommand{\Qm}{\mathcal{Q}}

\renewcommand{\aa}{\mathfrak{a}}

\renewcommand{\gg}{\mathfrak{g}}
\newcommand{\kk}{\mathfrak{k}}

\renewcommand{\tt}{\mathfrak{t}}

\newcommand{\doi}[1]{\textsc{doi}: \href{http://dx.doi.org/#1}{\nolinkurl{#1}}}

\DeclareMathOperator{\Ad}{Ad}
\DeclareMathOperator{\ad}{ad}

\DeclareMathOperator{\End}{End}
\DeclareMathOperator{\Aut}{Aut}
\DeclareMathOperator{\grad}{grad}
\DeclareMathOperator{\Hess}{Hess}

\DeclareMathOperator{\Lie}{Lie}

\DeclareMathOperator{\tr}{tr}
\DeclareMathOperator{\id}{id}
\DeclareMathOperator{\rk}{rk}
\DeclareMathOperator{\Stab}{Stab}
\DeclareMathOperator{\Norm}{Norm}

\title[Fibering polarizations, Mabuchi rays and Peter--Weyl theorem]{Quantization in fibering polarizations, Mabuchi rays and \\ geometric Peter--Weyl theorem}
\author{T.~Baier}
\address{T.~Baier\\Center for Mathematical Analysis, Geometry and Dynamical Systems, Instituto Superior T\'ecnico, Lisbon, Portugal.}
\email{thomas.baier@tecnico.ulisboa.pt}

\author{J.~Hilgert}
\address{J.~Hilgert\\Department of Mathematics, Paderborn University, Paderborn, Germany.}
\email{joachim.hilgert@upb.de}

\author{O.~Kaya}
\address{O.~Kaya\\Department of Mathematics, Galatasaray University, Istanbul, Turkey.}
\email{oguzabel@gmail.com}

\author{J.~M.~Mourão}
\address{J.~M.~Mourão\\Department of Mathematics and Center for Mathematical Analysis, Geometry and Dynamical Systems, Instituto Superior T\'ecnico, Lisbon, Portugal.}
\email{jmourao@tecnico.ulisboa.pt}

\author{J.~P.~Nunes}
\address{J.~P.~Nunes\\Department of Mathematics and Center for Mathematical Analysis, Geometry and Dynamical Systems, Instituto Superior T\'ecnico, Lisbon, Portugal.}
\email{jpnunes@tecnico.ulisboa.pt}

\date{\today}

\begin{document}

\begin{abstract}
In this paper we use techniques of geometric quantization to give a geometric interpretation of the Peter--Weyl theorem. We present a novel approach to half-form corrected geometric quantization in a specific type of non-Kähler polarizations and study one important class of examples, namely cotangent bundles of compact semi-simple groups $K$. Our main results state that this canonically defined polarization occurs in the geodesic boundary of the space of $K\times K$-invariant Kähler polarizations equipped with Mabuchi's metric, and that its half-form corrected quantization is isomorphic to the Kähler case. An important role is played by invariance of the limit polarization under a torus action. 

Unitary parallel transport on the bundle of quantum states along a specific Mabuchi geodesic, given by the coherent state transform of Hall, relates the non-commutative Fourier transform for $K$  with the Borel--Weil description of irreducible representations of $K$.
\end{abstract}

\maketitle

\tableofcontents

\section{Introduction}
\label{sect-introd}

On the cotangent bundle $T^\ast K$ of a compact semi-simple Lie group $K$, the space of $K\times K$-invariant Kähler structures compatible with the standard symplectic form $\omega_{\rm std}$ form an infinite-dimensional space $\cM$ which is effectively parameterized by Weyl-invariant convex functions $g$ on the dual $\tt^\ast$ of the Lie algebra of a maximal torus $T\subset K$. Such a function determines a complex structure $J_g$, or equivalently a polarization (in the sense of  geometric quantization) $\Pm_g$, which in turn define a bundle of quantum spaces $\HH \to \cM$, where for any point $g \in \cM$, $\HH$ is constituted by the space of $J_g$-holomorphic sections of the line bundle $L\otimes \KK^{1/2}$, where $L$ is the prequantum line bundle and $\KK^{1/2}$ is a square root of the canonical bundle, both of which are trivializable in this case.

In the first part of the present article, we introduce a polarization $\Pm_{\mathrm{KW}}$, whose first description and proof of existence is due to Kirwin and Wu \cite{kirwin.wu:manuscript}; it is not Kähler, but rather of ``mixed" type. Following a more general and novel approach, we start from the diagram defined naturally in terms of the Hamiltonian actions described in Section \ref{sec:equiv-geom}, 
\[
\begin{tikzcd}[column sep=small]
 & T^\ast K \ar[ld, "\mu_{\mathrm L} \times \mu_{\mathrm R}" swap] \ar[rd, "\mu_{\mathrm{inv}}"] & \\
 \kk^\ast \underset{\tt^\ast_+}{\times} \kk^\ast \ar[rr, "\phi"] & & \tt^\ast_+
\end{tikzcd}
\]
The Kirwin--Wu polarization is defined by equipping the fibers of $\phi$ over the interior $\breve{\tt}^\ast_+$ of the positive Weyl chamber with the unique $K\times K$-invariant complex structures which are Kähler for the reduced symplectic forms.

Our first main theorem states that the Kirwin--Wu polarization occurs on the boundary of the space of Kähler polarizations $\cM$:

\begin{theorem*}[see Thm.~\ref{thm_convpol} below]%\label{main-thm 1}
Assume $h$ is strictly convex and Weyl-invariant; then the polarization defined by $g+th$ converges to $\Pm_\mathrm{KW}$
  \[
   \Pm_{g+th} \overset{t\to\infty}{\longrightarrow} \Pm_\mathrm{KW}
  \]
pointwise in the Lagrangian Grassmannian of the complexified tangent bundle of the regular stratum $(T^\ast K)_\mathrm{reg}$ (see Section \ref{prelimsingtorus}).
\end{theorem*}

This convergence result leads to the following schematic picture of our parameter space of $K\times K$-invariant polarizations on $T^\ast K$:
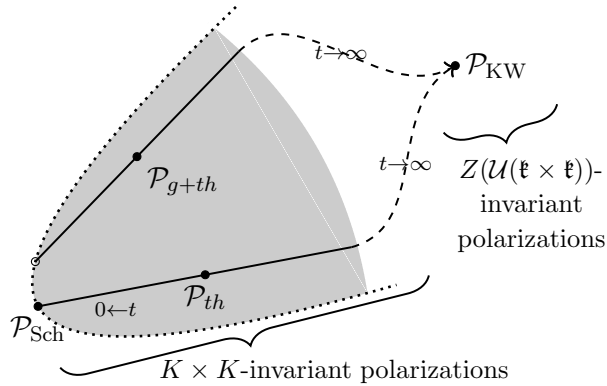
\[
\begin{tikzpicture}[scale=1.5]
  \begin{axis}[axis lines=none,axis equal,grid=both,no marks,domain=-2:2,xmax=2.5,xmin=-2.5,ymax=7,ymin=-1,samples=100,rotate=-60,]
  \addplot[dotted, line width=1pt,name path=a] {(x^2)};
  \addplot[dotted, line width=0pt,name path=b] {(4-0.1*x^2)};
  \addplot[fill=none] fill between[of=a and b,split,
    every segment no 1/.style={fill,gray,opacity=.4},] ;
  \draw [fill, black] (axis cs: 0,0) circle [radius=1pt] node[below,name=PSch] {$\Pm_{\mathrm{Sch}}$};
  \draw [fill, black] (axis cs: 0,6) circle [radius=1pt] node[right,name=PKW] {$\Pm_{\mathrm{KW}}$};
  \draw [fill, black] (axis cs: 0.7,2) circle [radius=1pt] node[below,name=Pth] {$\Pm_{th}$};
  \draw [fill, black] (axis cs: -1,2) circle [radius=1pt] node[below right,name=Pgth] {$\Pm_{g+th}$};
  \draw (axis cs: -0.5,0.25) circle [radius=1pt];
  \draw[line width=0.75pt] (axis cs: 0,0) to node[pos=.25, below] {${\tiny {}_{0 \leftarrow t}}$} (axis cs: 1.313,3.75);
  \draw[line width=0.75pt,->,dashed] (axis cs:1.313,3.75) to [out=75,in=270] node[pos=.5] {${\tiny {}_{t \to \infty}}$} (axis cs:0,6);
  \draw[line width=0.75pt] (axis cs: -0.5,0.25) to (axis cs: -1.5,3.75);
  \draw[line width=0.75pt,->,dashed] (axis cs:-1.5,3.75) to [out=105,in=270] node[pos=.5] {${\tiny {}_{t \to \infty}}$} (axis cs:0,6);
\draw[decorate, line width=0.75pt, decoration={brace, amplitude=2ex, raise=1ex}] (axis cs: 1.95,4.45) -- (axis cs: 0.65,-0.1) node[pos=.25, below=5ex] {{\small $K\times K$-invariant polarizations}};  
\draw[decorate, line width=0.75pt, decoration={brace, amplitude=2ex, raise=1ex}] (axis cs: 0.9,6.9) -- (axis cs: 0.5,5.5) node[pos=.2, below=4ex, text width=2.5cm, text centered] {{\small $Z(\mathcal U(\kk \times \kk))$-$\qquad$ invariant polarizations}};
  \end{axis}
\pgfresetboundingbox
\path[use as bounding box] (3,-2) rectangle (10.5,1.75);
\end{tikzpicture}
\]
The real Schrödinger polarization $\Pm_\mathrm{Sch}$ (also called the vertical polarization) and mixed Kirwin--Wu polarization $\Pm_\mathrm{KW}$ occur at the boundary (the former at finite geodesic time, the latter at infinity) of the convex set of invariant Kähler polarizations, represented by the gray area.  The Schrödinger polarization has been studied in detail in \cite{hall:2002}.
The Mabuchi geodesic generated by the squared norm of the moment map and the associated coherent state transform of Hall will turn out to implement the non-commutative Fourier transform as an object in geometric quantization.

Given the occurrence of the Kirwin--Wu polarization as continuous degeneration of Kähler polarizations, it is natural to extend the bundle of quantizations
\[
 \begin{tikzcd}
 \HH \ar[d] \ar[r, phantom, "\subset"] & \overline{\HH} \ar[d] \\
 \cM \ar[r, phantom, "\subset"] & \overline{\cM} = \cM \cup \{ \Pm_\mathrm{Sch},\Pm_\mathrm{KW}\}
 \end{tikzcd}
\]
A preliminary but central point of our endeavor is the novel uniform definition we give (in Section \ref{sss_defquan}) of half-form corrected quantizations in Kähler and mixed polarizations. In Section \ref{sec_fibpol} we introduce a specific class of ``fibering" polarizations whose local structure is quite rigid (cf. Diagram (\ref{diag_stdpol}) and Proposition~\ref{prop_mp_props}), and permits the definition of Bohr--Sommerfeld conditions. The role of these conditions is elucidated by the ``geometric'' structure the quantizations carry, as we establish in our second main result.
\begin{theorem*}[See Thm. \ref{thm_BS1}] The quantum states associated to a fibering polarization are supported on the subset of Bohr--Som\-mer\-feld leaves.
\end{theorem*}

Our next main result returns to our main example of a fibering polarization, the Kirwin--Wu polarization $\Pm_\mathrm{KW}$, in the context of the invariant Mabuchi space. The extension of the bundle of quantizations on $\cM$ to the boundary points in $\overline{\cM}$ is $K\times K$-equivariant. Considering the quantum states $s^{g}_{\lambda,A}$ corresponding to a matrix element of the complexification (indexed by a highest weight $\lambda$ and an endomorphism of the respective representation $A \in \End V_\lambda$) we find:
\begin{theorem*}[See Thm.~\ref{thm_convstates}] For the geodesic ray $g_t := g+th$, on the regular subset $(T^\ast K)_{\mathrm{reg}}$
the family of quantum states $s^{g_t}_{\lambda,A}$ extends continuously to the Kirwin--Wu polarization,
\[
 \lim_{t \to \infty} s^{g_t}_{\lambda,A} =: s^\mathrm{KW}_{\lambda,A} .
\]
\end{theorem*}

In Section \ref{sect-cst}, we interpret Theorem \ref{thm_convstates} in terms of generalized coherent state transforms (gCST). Section \ref{sect-connection} describes how the imaginary time flow of $h$ lifts to the extended bundle of quantum states by means of a flat connection whose parallel transport is given by the gCST.  In the case when $h$ is the quadratic Casimir, the coherent state transform of Hall defines the parallel transport for a unitary connection on the restriction of the extended bundle of quantum states to the Mabuchi geodesic ray generated by the quadratic Casimir. This extends  the 
results of \cite{hall:2002, florentino.matias.mourao.nunes:2005, florentino.matias.mourao.nunes:2006, 
Huebschmann08,
kirwin.mourao.nunes:2013}, where quantization in the Schr\"odinger and invariant K\"ahler polarizations were related by generalized coherent state transforms, to infinite geodesic time. In particular, along the Mabuchi geodesic generated by the quadratic Casimir one gets a unitary $K\times K$-invariant identification between the Hilbert spaces of quantum states. This corresponds to a natural geometric interpretation of the operator-valued Fourier transform of $f\in L^2(K,dx)$ described by the Peter--Weyl expansion of $f$, where $dx$ is the normalized Haar measure on $K$. Each term in this  expansion of $f$ is associated with a geometric Bohr--Sommerfeld cycle and a distributional section supported on it. The correspondence is unitary with respect to a natural inner product structure on the space of ${\mathcal P}_{\rm KW}$-polarized sections. 
The inner product is determined up to scalar on each Bohr--Sommerfeld cycle by its invariance properties. The chosen normalization is obtained as the limit of the inner products for finite geodesic time and coincides with the inner product predicted by the Peter--Weyl theorem.

The distributional sections associated to the Kirwin--Wu polarization can be directly interpreted in terms of holomorphic sections for Borel--Weil line bundles on coadjoint orbits. Thus the limit of the generalized coherent state transforms not only provides a geometric interpretation of the operator-valued Fourier transform but it also provides a direct link between the Peter--Weyl theorem and the Borel--Weil theorem, as described in Section \ref{sect-fourier}.   Section \ref{sec_conjectures} contains the outline of a program that aims at linking representations and geometric cycles in the context of Hamiltonian $G$-spaces.

While this paper was under preparation, the preprint \cite{leung.wang:2022} appeared where mixed polarizations arising at infinite Mabuchi geodesic time are also considered; there, the Hamiltonians generating the Mabuchi geodesics are convex only along a subspace of momentum variables of a torus action. In the toric setting, similar results were studied in  \cite{pereira:2022} and will appear in \cite{mourao.nunes.pereira:2023}. We note, however, that in the present paper, while we are in the context of non-abelian Hamiltonian actions, the initial K\"ahler polarization is not invariant by the action of the invariant torus $T_\mathrm{inv}.$

The present work gives an extension of the orbit method in the sense that the whole set of irreducible representations of $K$ is treated at once in terms of a quantization of the symplectic manifold $T^*K$. Indeed, as described in this paper, the collection of quantizable symplectic reductions of $\mathcal{P}_\mathrm{KW}$ 
is in bijection with the set of highest weights in $\hat K$, via the correspondence
$\lambda \mapsto \mathcal{O}_{\lambda+\rho}\times \mathcal{O}_{(\lambda+\rho)^*}$.
The orbit method has been very successful in providing parameters and geometric realizations of irreducible representations for various classes of Lie groups (see e.g. \cite{Duflo82}). There are also appealing descriptions of branching rules such as Plancherel measures in symplectic terms (see e.g. \cite{Lipsman80,Vergne82}). Typically, however, such decompositions are not proved via the orbit method but given as reinterpretations of results proven by other methods. Intertwining operators are rarely constructed via the geometric quantization. The main exception so far are the ones obtained by the BKS (Blattner--Kostant--Sternberg) transformations related to a change of polarization (see e.g. \cite{florentino.matias.mourao.nunes:2006,Huebschmann08,Lisiecki87}). In particular, Fourier transforms of non-commutative groups and associated symmetric spaces have been constructed via geometric quantization only in a few cases such as Riemannian symmetric spaces associated with complex semisimple groups \cite{Lisiecki87}. Now, the limits of the generalized coherent state transforms actually give such a construction of the operator-valued Fourier transform (see \cite[\S~18.8.1]{Dixmier77}) of compact Lie groups purely in terms of geometric quantization.

In future work, we will apply our methods to compact symmetric spaces and we expect that they can also be extended to apply to Riemannian symmetric spaces of non-compact type. In order to deal with non-compact semisimple groups one would have to establish a further extension of the method to non-Riemannian symmetric spaces. The program outlined in Section \ref{sec_conjectures} will also be addressed in future work.

\section{Preliminaries}
\label{sect-prelim}

\subsection{Lie theoretic notations and conventions}\label{prelimLie} Throughout we consider a compact connected and simply connected\footnote{The simply-connectedness condition is just for simplicity of exposition.} Lie group $K$, together with a fixed maximal torus $T\subset K$. Besides the inclusion of the corresponding Lie algebra $\tt \subset \kk$, there is also a canonical way to identify the dual $\tt^\ast$ with the subspace $(\kk^\ast)^T\subset \kk^\ast$ invariant under the restriction of the coadjoint action to the maximal torus.
In this way, both the Lie algebra and its dual split canonically and $T$-equivariantly,
\begin{equation}\label{Lie_decomp}
\kk = \tt \oplus \tt^\perp , \qquad
\kk^\ast = \tt^\ast \oplus \tt^{\ast \perp} .
\end{equation}
When we occasionally use a $K$-invariant inner product to identify $\kk$ and $\kk^\ast$ below, the identification is block-diagonal with respect to these decompositions.

Denote by $\tt_\ZZ \subset \tt$ the kernel of the exponential map $\tt \to T$, and by $\tt^\ast_\ZZ$ its integral dual,
\[
 \tt^\ast_\ZZ := \{ \xi \in \tt^\ast\mid  \forall \eta \in \tt_\ZZ: \langle \eta, \xi \rangle \in \ZZ \} .
\]
Note that $2\pi i \tt^\ast_\ZZ$ is naturally identified with the weight lattice via the identification
\[
 \lambda \in \tt^\ast_\ZZ
 \qquad \leftrightarrow \qquad
 \chi_\lambda:T \to U(1),
 \chi_\lambda \circ \exp(\eta) := e^{2\pi i \langle \lambda, \eta \rangle} .
\]

We equally fix a choice of convex fundamental domain $\tt^\ast_+$ of the Weyl group action on $\tt^\ast$, and denote its interior by $\breve{\tt}^\ast_+$. It is canonically isomorphic to the positive Weyl chamber $i \tt^\ast_+ \subset i \tt^\ast$.
Due to these canonical identifications we will also refer to $\tt_+^\ast$ as positive Weyl chamber and to the elements of $\tt_\ZZ^\ast$ as weights.

As is customary, elements of the Weyl group $w \in W := \Norm_K (T) / T$ will occasionally be confounded with lifts in $K$, and the longest element will be denoted $w_0$.

\subsubsection{Line bundles and their sections} $K$-linearized (also called $K$-ho\-mo\-ge\-neous) line bundles on $K/T$ are in bijective correspondence with characters $\chi: T \to U(1)$ via $L_{\chi} := K \overset{T}{\times} \CC_{\chi^{-1}}$, where we use the notation
\[
 [k,z] = [kt,\chi(t^{-1})z] \in K \overset{T}{\times} \CC_{\chi}
\]
for the bundle associated to the principal bundle $K \to K/T$ and the representation $\chi^{-1}$. In particular, smooth sections of $L_\chi$ are in bijection with equivariant functions $f:K\to \CC$ such that $f(kt)=\chi(t)f(k)$.

\subsubsection{Borel subgroups of $K_{\CC}$ and the Borel--Weil Theorem} Consider any highest weight representation $\pi_\lambda:K \to \End(V_\lambda)$ with regular weight $\lambda$, and consider a highest weight vector $v_\lambda \in V_\lambda$. We collect some well-known facts about the relation between invariant complex structures, Borel subgroups, and line bundles on $K/T$ (see e.g. \cite{serre:1959}):

\begin{theorem}[Borel--Weil]\label{thm_BorelWeil}
\item[(a)] The stabilizer $B := \Stab_{K_{\CC}}[v_\lambda]$ of the highest weight vector in the projectivization $[v_\lambda]\in \PP V_\lambda$ is a Borel subgroup of the complexification $K_{\CC}$.
\item[(b)] The $K$-equivariant embedding
\begin{equation}\label{OO_emb}
K/T \ni kT \overset{\psi}{\mapsto} [\pi_\lambda(k)v_\lambda] \in K\cdot [v_\lambda] \subset \PP V_\lambda
\end{equation}
is a diffeomorphism onto its image, and it is a biholomorphism with respect to the invariant complex structure given by the identification
\[
 K/T = K_{\CC}/B .
\]
\item[(c)] The pull-back of the hyperplane bundle $\OO_{\PP V_\lambda}(1)$ along the embedding $\psi$ is $L_{\chi_\lambda}$.
\item[(d)] Every line bundle $L_{\chi_\lambda}$ is holomorphic on $K_{\CC}/B$ in a unique way, since $L_\lambda:=L_{\chi_\lambda} = K_{\CC}\overset{B}{\times} \CC_{\chi_{-\lambda}}$, and its space of holomorphic sections is isomorphic to the dual representation
\[
 H^0(K_{\CC}/B , L_{\lambda}) \cong V_\lambda^\ast .
\]
Explicitly, these sections are identified as the $T$-equivariant maps
\[
 f_v(k) := \langle \pi_\lambda(k)v_\lambda, v \rangle \in \CC
 \qquad \forall v \in V_\lambda.
\]
\end{theorem} 
\begin{remark} 
\item[(i)] More generally, all Borel subgroups $B^w \subset K_{\CC}$ containing $T_{\CC}$ arise similarly as stabilizers of the extremal weight vectors $[\pi_\lambda(w)v_\lambda]$ as $w$ varies over the Weyl group. For brevity, we will denote the Borel opposite to $B$, that is, the conjugate by a longest element $w_0$ of the Weyl group, by $B_- := B^{w_0}$. It is the stabilizer of a lowest weight eigenspace $[v_{w_0 \lambda}]$.

\item[(ii)] We implicitly used the fact that the characters of the maximal torus $T$ extend to the Borel subgroup $B$ (whose maximal unipotent subgroup we denote by $U$) through the identification
\begin{equation}\label{eq_Borelchar}
 B \longrightarrow B / U = T_{\CC} \overset{\chi_{\lambda}}{\longrightarrow} \CC^\ast ,
\end{equation}

\item[(iii)] Of course (c) and (d) are also valid for weights in the walls of the positive Weyl chamber, whereas (a) and (b) would have to be modified using more general parabolic subgroups.
\end{remark}

\subsection{Equivariant geometry of $T^\ast K$}\label{sec:equiv-geom} Here, we recall some aspects of the geometry of the cotangent bundle $T^\ast K$ of a compact connected Lie group $K$; see for instance \cite[ch. IV.4]{liebermann.marle:1987} or \cite{alekseevsky.grabowski.marmo.michor} for more complete accounts. 
\subsubsection{Symplectic geometry} The cotangent bundle $T^\ast K$ is naturally identified with the product $K\times \kk^\ast$,
\begin{equation}\label{identification}
  K \times \kk^\ast \ni (x,\xi) \mapsto  \xi \circ d_e L_{x^{-1}}\in T^\ast_x K ,
\end{equation}
(where $L_x$ denotes left translation by $x\in K$) and 
we will use this identification tacitly throughout. The derivative of the isomorphism (\ref{identification}) induces an identification of tangent spaces
\[
T_{\left( L_{x^{-1}} \right)^\ast \xi}\left(T^\ast K\right) = \kk\times\kk^\ast
\]
under which the standard symplectic form at tangent vectors $(\dot x_i, \dot \xi_i) \in \kk \times \kk^\ast$ at the point $(x,\xi)$ is
\begin{equation}\label{symp_form}
\omega_\mathrm{std}\left( (\dot x_1, \dot \xi_1 ),(\dot x_2,\dot \xi_2) \right) =
 \langle \dot \xi_2 ,\dot x_1 \rangle - \langle \dot \xi_1 ,\dot x_2 \rangle + \langle \xi, [\dot x_1, \dot x_2] \rangle.
\end{equation}
In particular, the corresponding volume form becomes identified with the product measure
\begin{equation}\label{vol_form}
 \int_{T^\ast K} f \frac{\omega_{\mathrm{std}}^n}{n!} = \int_K \int_{\kk^\ast} f(x,\xi) dx d\xi ,
\end{equation}
where $dx$ denotes the normalized Haar measure on $K$, and $d\xi$ the corresponding Lebesgue measure on $\kk^\ast$ \cite[(2.3)]{hall:2002}.
The natural (left) $K\times K$-action on $T^\ast K$ is identified with
\[
 (x_1,x_2)(x,\xi) = (x_1 x x_2^{-1},\Ad_{x_2}^\ast \xi) ,
\]
and it is Hamiltonian with moment map $\mu=\mu_L\times\mu_R:T^\ast K \to \kk^\ast\times\kk^\ast$
\[
 K \times \kk^\ast \ni (x,\xi) \mapsto (\mu_L\times\mu_R)(x,\xi) = (\Ad_x^\ast \xi,-\xi) \in \kk^\ast \times \kk^\ast .
\]
\subsubsection{$KAK$-decomposition and the singular torus action of the invariant moment map}\label{prelimsingtorus}

The \emph{regular set} $\kk^\ast_\mathrm{reg}$ consists of the elements of the dual of the Lie algebra whose stabilizer is a maximal torus; it is the union of coadjoint orbits through elements in the interior $\breve{\tt}^\ast_+$ of the positive Weyl chamber. Its product with $K$ defines a subset $K \times \kk^\ast_\mathrm{reg} \cong \left(T^\ast K\right)_\mathrm{reg} \subset T^\ast K$, to which we refer as the \emph{regular stratum}. On it we consider the \emph{$KAK$-decomposition}, that is, the map
\begin{equation}\label{KAK_deco}
  \begin{tikzcd}[row sep=0pt]
    K\times \breve{\tt}^\ast_+ \times K \ar[r] & K\times \kk^\ast_{reg} \\
    (x_1,\xi_+,x_2) \ar[r,mapsto] & (x_1 x_2^{-1}, \Ad_{x_2}^\ast \xi_+)
\end{tikzcd}
\end{equation}
which is a $K\times K$-equivariant principal $T$-bundle with respect to the action $(x_1,\xi_+,x_2)\cdot t = (x_1 t,\xi_+,x_2t)$. We will frequently use the $KAK$-decomposition somewhat implicitly by writing $(x,\xi) = (x_1 x_2^{-1}, \Ad_{x_2}^\ast \xi_+)$, the trivial check that the result of the calculations is not affected by the indeterminacy of the $x_i$ being understood.

The \emph{invariant moment map} is defined by composing $\mu$ with the inverse of the natural homeomorphism $\tt^\ast_+ \times -\tt^\ast_+ \cong (\kk^\ast\times\kk^\ast)/(K\times K)$. In our setting, it is naturally identified with the map $\mu_\mathrm{inv}:T^\ast K \to \tt^\ast_+$
\begin{equation}\label{inv_mm}
 \mu_\mathrm{inv}\left((x,\xi)=(x_1x_2^{-1},\Ad_{x_2}^\ast \xi_+)\right) = \xi_+ .
\end{equation}
Here, for reasons related to the degenerations of complex structures we consider and which will become apparent later, it is most convenient to distinguish the positive Weyl chamber
\begin{equation}\label{KxK_Weyl}
  \tt^\ast_+ \times -\tt^\ast_+ \quad \subset \quad \tt^\ast \times \tt^\ast 
\end{equation}
with respect to the maximal torus $T\times T \subset K\times K$ of the group acting on $T^\ast K$.

The \emph{Kirwan polytope} \cite{kirwan:1984}, which by definition is the intersection of this positive Weyl chamber with the image of $\mu$, is therefore an anti-diagonally embedded copy of $\tt^\ast_+$, hence it is naturally identified with the image of $\mu_\mathrm{inv}$ via $\xi_+ \mapsto (\xi_+,-\xi_+)$. The moment map image of the regular stratum is naturally identified with the fibered product $\kk^\ast \times_{\breve{\tt}^\ast_+} \kk^\ast \subset \kk^\ast \times \kk^\ast$ with respect to the maps $\xi \mapsto \Ad_K^\ast \xi \cap \tt^\ast_+$ and $\xi \mapsto -\Ad_K^\ast \xi\cap \tt_+^\ast$, respectively.

The inverse image of the relative interior of the Kirwan polytope under the moment map
\[
  \Sigma := \mu^{-1}\left( \breve{\tt}^\ast_+ \right) = \left\{ (t,\xi_+)\mid  t\in T, \xi_+ \in \breve{\tt}^\ast_+ \right\} \subset K \times \kk^\ast
\]
is an open symplectic subvariety of $T^\ast K$ called the \emph{symplectic cross section}. In our case, it is just the open toric variety $T\times \breve{\tt}^\ast_+$, but notice that due to the sign convention in (\ref{symp_form}) the torus action is given by $s\cdot (t,\xi_+)=(s^{-1}t,\xi_+)$.

Since $\mu_\mathrm{inv}$ takes value in the dual of the Lie algebra of the maximal torus $T$, it is a rather natural idea (which goes back at least to Guillemin and Sternberg \cite{guillemin.sternberg:1984}) to interpret it as the moment map of a torus action; indeed this is possible on a dense open subset in great generality (cf. also \cite{lane:2017} and \cite[Lemma 2.3]{knop:2011}). We content ourselves with an explicit description in our setting:
\begin{proposition}
  Consider a fixed copy $T_\mathrm{inv}:=T$ of the maximal torus $T\subset K$; the restriction of the invariant moment map to the regular stratum
  \[
  \mu_\mathrm{inv}:(T^\ast K)_\mathrm{reg} \to \breve{\tt}^\ast_+
  \]
  is the Hamiltonian moment map of the $T_\mathrm{inv}$-action
\begin{equation}\label{sing_action}
 (x = x_1 x_2^{-1},\xi = \Ad_{x_2}^\ast \xi_+) \star t := (x x_2 t^{-1} x_2^{-1},\xi) .
\end{equation}
It equips the restriction of the moment map $\mu$ to the regular stratum
\[
 (T^\ast K)_\mathrm{reg} \overset{\mu}{\longrightarrow} \kk^\ast \times_{\breve{\tt}^\ast_+} \kk^\ast
\]
with the structure of a $K\times K$-equivariant $T_\mathrm{inv}$-principal bundle over its moment map image.
\end{proposition}
\begin{proof}
This is easy to verify, since this action commutes with the $K\times K$-action, and over the symplectic cross section $\Sigma$ it coincides with (\ref{sing_action}).
\end{proof}

\begin{remark} \item[(i)] We refer to this action as the ``singular torus action'', and to $T_\mathrm{inv}$ as the ``singular torus'', although there is nothing singular about it except for the fact that the action is only defined on the regular stratum and does not extend continuously to all of $T^\ast K$; there are criteria to determine which functions of the invariant moment map extend smoothly outside of the regular stratum for multiplicity-free manifolds, see \cite[Thm. 4.1]{knop:2011}.

 \item[(ii)] The map from the $KAK$-decomposition (\ref{KAK_deco}) becomes equivariant with respect to $T_{\mathrm{inv}}$ if we set
  \[
    (x_1, \xi_+, x_2)\star t := (x_1 t^{-1}, \xi_+, x_2) .
  \] 
\end{remark}

\subsection{Invariant Kähler structures and Mabuchi geodesic rays}\label{prelimKahler}

\subsubsection{Invariant Kähler structures}
In the sequel, the Legendre transform between dual vector spaces will play an important role, and we use the opportunity to introduce the necessary notation: for any smooth function $f:V\to\RR$ on a vector space $V$ we identify $d_v f$ with an element of the dual vector space $V^\ast$ via the canonical isomorphisms
\[
 \begin{tikzcd}
  T_v V \arrow[r, "d_v f"] \arrow[d, equal, swap, "\,\text{can.}"] & T_{f(v)} \RR \arrow[d, equal, "\,\text{can.}"] \\
 V \arrow[r, "d_v f"] & \RR
 \end{tikzcd} 
\]
Globally, these derivatives define the \emph{Legendre transform} (of particular interest for convex or concave functions $f$)
\[
\LL_f:V \to V^\ast, \quad  \LL_f(v) = d_v f .
\]

\begin{proposition}[{\cite{neeb:2000a,neeb:2000b,kirwin.mourao.nunes:2013}}]\label{prop_Pg} Let $g:\kk^\ast \to \RR$ be a strictly convex function which is invariant under the coadjoint action; then
  \item[(a)] the Legendre transform $\LL_g: \kk^\ast \to \kk$ is $K$-equivariant, and restricts to the Legendre transform $\LL_{(g\vert _{\tt^\ast})}$ of the restriction $g\vert_{\tt^\ast}$,
    \[
    \begin{tikzcd}
      \kk^\ast \ar[r, "\LL_g"] & \kk \\
      \tt^\ast \ar[u, phantom, "\bigcup"] \ar[r, "\LL_{(g\vert_{\tt^\ast})}"] & \tt \ar[u, phantom, "\bigcup"] .
    \end{tikzcd}
    \]
\item[(b)] if $\LL_g$ is surjective, for instance if $g$ is uniformly convex, the map
  \begin{equation}
    T^\ast K \ni (x,\xi) \mapsto x e^{i d_\xi g} \in K_{\CC}
  \end{equation}
  is a diffeomorphism onto the complexification $K_{\CC}$ such that the pull-back $J_g$ of the natural complex structure along this map defines a Kähler structure $(T^\ast K,\omega_\mathrm{std},J_g)$.

  Furthermore, the ring of global $J_g$-holomorphic functions is generated by the matrix coefficients
\begin{equation}\label{flambdaA}
f^g_{\lambda,A}(x,\xi) :=\tr (\pi_\lambda (xe^{id_\xi g})A),
\end{equation}
where $\lambda$ ranges over the dominant integral weights, $\pi_\lambda:K \to \Aut (V_\lambda)$ is the corresponding highest weight representation of $K$, and $A\in \End(V_\lambda)$.
\item[(c)] The (global) K\"ahler potential (for $\omega_\mathrm{std}$ with respect to $J_g$) is \cite[Thm. 3.2]{kirwin.mourao.nunes:2013}
\begin{equation}\label{kappa_g}
\kappa_{g}(x,\xi) :=  \langle \xi, d_\xi g \rangle - g ( \xi ) .
\end{equation}
\end{proposition}
The statement on the restriction in (a) follows from equivariance, which implies that $\LL_g(\tt^*)\subseteq \mathfrak z_\mathfrak k(\tt)=\tt$.

\begin{definition}
We will denote the space of $K\times K$-invariant Kähler polarizations described in (b) by $\cM$, and identify it with the space of Weyl-invariant convex functions on $\tt^\ast$ whenever convenient. The K\"ahler polarization associated to the invariant function $g$ will be denoted by $\Pm_g.$
\end{definition}

We will also use the following properties from \cite[Lemmas 3.4 and 3.6]{kirwin.mourao.nunes:2013}; here we fix a $K$-invariant inner product $\langle\ ,\ \rangle$ on $\kk^\ast$, and identify $\kk$ and $\kk^\ast$ $K$-equivariantly in the usual way via the Legendre transform $\LL_Q$ associated to the norm square $Q(\xi) = \frac{1}{2} \langle \xi,\xi \rangle$.
\begin{proposition}\label{p-1}
  For arbitrary smooth $K$-invariant functions $g, h \in C^\infty(\kk^\ast)^K$,  we have
 \item[(a)] If $\xi\in \kk^\ast_\mathrm{reg}$ is regular, then $[ d_\xi g, d_\xi h] = 0$.
 \item[(b)] The invariant Hamiltonian $h\circ \mu_\mathrm{R} = h \vert_{\tt^\ast_+}\circ \mu_\mathrm{inv}$ generates the Hamiltonian vector field and flow, respectively,
 \[
  \left(X_{h \circ \mu_\mathrm{R}}\right)_{(x,\xi)} = \left( (L_x)_\ast  d_\xi h,0 \right) ,
\quad \text{ and } \quad
  \varphi_t^{X_{h\circ \mu_\mathrm{R}}} (x,\xi) = (x e^{t d_\xi h},\xi) ,
  \]
  where $(L_x)_\ast$ is the push-forward by the left translation $L_x$.
\item[(c)] Let $\Hess_h$ denote the Hessian of $h$, which can be defined by
\[
\frac{d}{dt}_{\vert_{t=0}}  d_{\xi+tb}h(a) = \langle b, \Hess_h(\xi) a\rangle_{\kk^\ast}, \,\, a,b\in \kk^\ast.
\]
Then,
\[
\forall\, u\in K:\quad \Hess_h(\Ad_u^* \xi) = \Ad_u^* \Hess_h(\xi) \Ad_{u^{-1}}^* .
\]
\item[(d)] If $h$ is strictly convex we have, as endomorphisms of $\kk^\ast$,
$$
\ad_{\xi^*}^* = \Hess_{h}^{-1}(\xi) \, \ad_{d_\xi h}^* ,
$$
where $\xi^*\in\kk$ is the Lie algebra element corresponding to $\xi$ under the isomorphism $\kk\cong \kk^\ast$ induced by the invariant inner product on $\kk$.
\item[(e)] Consider a Chevalley basis for $\kk\otimes \CC$ associated to the choice of maximal torus $T\subset K$ and positive Weyl chamber $\tt_+^*$. With respect to the dual basis of $\kk^*\otimes \CC$ 
and under the analog of the decomposition (\ref{Lie_decomp}), the $\CC$-linear extension to $\kk\otimes\CC$ of the Hessian of $h$ at a point $\xi \in \tt^\ast$, decomposes as
  \[
    \Hess_h(\xi) = \begin{bmatrix} \Hess_h(\xi)_{\vert_{\tt^\ast}} & 0 \\ 0 & \mathrm{diag} \frac{\langle\alpha,d_\xi h\rangle}{\langle\alpha , \xi\rangle_{\tt^*}} \end{bmatrix} ,
    \]
with $\alpha$ ranging over the set of roots.
\end{proposition}

\subsubsection{Mabuchi geodesics}
On a compact K\"ahler manifold $(M,\omega,J)$, the space of K\"ahler forms in the K\"ahler class $[\omega]\in H^2(M)$ can described in terms of relative K\"ahler potentials as the set 
$$
\{\omega_\phi:= \omega+ \frac{i}{2}\partial\bar\partial \phi >0 \mid \phi\in C^\infty(M)\}.
$$
The Mabuchi metric (defined in \cite{mabuchi:1987} and studied independently also in the early references \cite{semmes:1992,donaldson:1999}) plays an important role in recent advances in K\"ahler geometry. A path of K\"ahler metrics $\omega_t = \omega+ \frac{i}{2}\partial\bar\partial \phi_t$ defines a \emph{Mabuchi geodesic} if
\begin{equation}\label{mab_eqn}
\ddot \phi_t -\frac{1}{2}\| \grad_t \dot \phi_t \|^2_t = 0,
\end{equation}
where $\grad_t$ and $\|\cdot\|_{t}$ denote the gradient and norm relative to the K\"ahler metric defined by $\omega_t$.
Alternatively, the \emph{Moser maps} $\psi_t\in {\rm Diff}\,(M)$ defined by the differential equation
\[
\frac{d}{dt}_{\vert_{t=0}} \psi_t (x) = -J X^{\omega_t}_{\dot \phi_t}(\psi_t(x)) 
\]
(where $X^{\omega_t}_{\dot \phi_t}$ is the Hamiltonian vector field of the time derivative $\dot \phi_t$ of the relative Kähler potential with respect to the Kähler form $\omega_t$) trivialize the family of Kähler forms in the sense that $\psi_t^*\omega_t=\omega$. In these terms, the geodesic equation takes the form
\begin{equation}\label{mab_eqn2}
  \psi_t^\ast \dot \phi_t = \dot\phi_0 .
\end{equation}
In any case, a path $\omega_t$ of cohomologous Kähler forms can therefore equivalently be described in the so-called symplectic picture, by a path of complex structures $J_t:=\psi_t^* J$ on $M$. Then the two K\"ahler structures $(\omega, J_t)$ and $(\omega_t, J)$ are equivalent. Also note that the geodesic equation (\ref{mab_eqn}) respectively (\ref{mab_eqn2}) in the space of relative Kähler potentials still makes sense on a non-compact manifold $M$.

Explicit solutions to the Mabuchi geodesic equation (which can also be formulated as a complex homogeneous Monge-Amp\`ere equation) are rare. Important examples are given by compact symplectic toric manifolds which are smooth compactifications of $(\mathbb C^*)^n\cong T^\ast  T^n.$ Just as in this abelian case, in the case of cotangent bundles of compact Lie groups the Mabuchi geodesic equation linearizes under Legendre transforms, as follows.

Let $g, h\in C^\infty(\kk^\ast)^K$ where both $g, h$ are uniformly convex, that is the operator norms of ${\rm Hess}_{g}, {\rm Hess}_{h}$ are bounded from below. We have the following 
\begin{proposition}(Section 10 of \cite{mourao.nunes:2015})
The path of $K\times K$-invariant K\"ahler structures $(T^\ast K,\omega_{\rm std},J_{g+th}), t>0,$ is a Mabuchi geodesic ray. 
\end{proposition}

The family of biholomorphisms of $T^\ast K$ corresponding to this geodesic is generated by the Hamiltonian flow of $X_h$ analytically continued to imaginary time $it, t>0$. Indeed, letting  
\begin{equation}\label{psit}
\psi_t :=  {\mathcal L}_{g}^{-1} \circ {\mathcal L}_{g+th}  : T^\ast K\to T^\ast K,
\end{equation}
we obtain
$J_{g+th}=\psi_t^* J_{g}$ and, explicitly (see \cite{kirwin.mourao.nunes:2013,mourao.nunes:2015}), for the generators of the coordinate rings ${\mathcal O}_{(T^\ast K,J_g)}$ and ${\mathcal O}_{(T^\ast K,J_{g+th})}$,
\[
\psi_t^* f_{\lambda,A}^g = e^{itX_h} f_{\lambda,A}^g = f_{\lambda,A}^{g+th}
\]
for $\lambda$ a dominant integral weight and $A$ an endomorphism of the respective irreducible representation, as in (\ref{flambdaA}).

\section{Half-form corrected quantization in fibering polarizations}\label{sec_fibpol}

\subsection{Fibering polarizations}\label{prelimmixedpol} Recall that a \emph{polarization} in the sense of geometric quantization is an integrable distribution $\Pm \subset T_{\CC}^\ast M$ in the complexified cotangent bundle of the symplectic manifold $(M,\omega)$ \cite{woodhouse:1991,Kir04}. The main examples are \emph{Kähler polarizations} induced by a compatible complex structure $J$, where $\Pm = T^{1,0}_J M$, and \emph{real polarizations} which are given by the complexification $\Pm = \LL\otimes\CC$ of a Lagrangian foliation $\LL \subset TM$. Beyond these most commonly encountered types, there are what we term \emph{mixed polarizations} in this paper, that is, complex polarizations with real directions; their local description is given by the generalized Newlander--Nirenberg theorem:
\begin{theorem}[{\cite{nirenberg:1958}}]\label{thm_nir}
An integrable complex distribution $\Pm \subset T_{\CC}M$ with real part $\Pm \cap \overline{\Pm}\cap TM$ of constant rank is locally generated by $k = \rk (\Pm \cap \overline{\Pm})$ real and $n-k = \frac{1}{2}\dim_{\RR} M - \rk (\Pm \cap \overline{\Pm})$ complex functions which Poisson-commute.
\end{theorem}
To a polarization $\Pm$ of this kind, one associates the isotropic distribution
   \[
     \DD := \Pm \cap \overline{\Pm} \cap TM ,
   \]
   whose integrability follows from that of $\Pm$. Its symplectic orthogonal
   \[
    \EE := (\Pm + \overline{\Pm}) \cap TM = \DD^{\perp_{\omega}}
   \]
   is fiberwise coisotropic but not in general integrable; we will restrict our attention to the case when both $\DD$ and $\EE$ are not only integrable, but defined by fibrations -- more precisely, since a global assumption of this kind would be too restrictive to be of great practical interest, we will assume that there is a commutative diagram
 \begin{equation}\label{diag_cored}
 \begin{tikzcd}[column sep=small, ampersand replacement=\&]
   \& \breve{M} \ar[ld, "\pi_{\DD}" swap] \ar[rd, "\pi_{\EE}"] \& \\
   B_{\DD} \ar[rr, "\phi"] \& \& B_{\EE} ,
 \end{tikzcd}
 \qquad \DD = \ker d\pi_{\DD}, \ \EE = \ker d\pi_{\EE} ,
 \end{equation}
 where $\breve{M} \subset M$ is an open dense subset and $\pi_{\DD}, \pi_{\EE}$ are submersions onto smooth manifolds $B_{\DD}, B_{\EE}$, respectively, which form a dual pair in the sense of Weinstein.
 \begin{remark} There is a natural condition on how to extend this relationship across the ``singular set'' $M\setminus \breve{M}$ using Ortega's ``singular dual pairs'', see Section \ref{sec_Fourierpol} below.
 \end{remark}
 The fibers of $\phi$ are the coisotropic reductions associated with $\EE$: that is to say there is a canonically defined vertical 2-form on $B_{\DD}$
   \[
   \omega_{\DD} \in \Gamma(B_{\DD},{\textstyle\bigwedge^2(\ker d\phi)^\ast)}
   \]
   which is symplectic on the fibers of $\phi$ and uniquely determined by the property that
   \[
     \pi_{\DD}^\ast \omega_{\DD} = \omega\vert_{\ker d\pi_{\EE}} .
   \]
 One has then the following geometric description of mixed polarizations (cf.~\cite[Prop.~5.4.7]{woodhouse:1991} or \cite[Prop.~1.5.5]{Kir04}).    
 \begin{proposition}\label{prop_polarizations}
   There is a one-to-one correspondence between mixed polarizations $\Pm$ on $M$ with real directions
   \[
     \DD = \Pm \cap \overline{\Pm} \cap TM
   \]
   and smooth integrable complex structures relative to $\phi$ compatible with $\omega_{\DD}$, i.e. $J_{\DD} \in \Gamma(B_{\DD},\End(\ker d\phi))$ such that
   \begin{itemize}
     \item[i)] $J_{\DD}^2 = -\id$,
     \item[ii)] $\omega_{\DD}(\cdot,J_{\DD}\cdot)$ is a pseudo-Riemannian metric on the fibers of $\phi$,
     \item[iii)] and $J_{\DD}$ is integrable on the fibers of $\phi$.
   \end{itemize}
 \end{proposition}
 \begin{definition} We call a polarization of this kind a \emph{fibering polarization} with respect to diagram (\ref{diag_cored}).
 \end{definition}
 
 \begin{remark} Fibering polarizations in particular include Kähler polarizations (for which $B_{\DD}=M$ and $B_{\EE}=\{pt\}$) and real polarizations induced from a Lagrangian fibration (for which $B_{\DD}=B_{\EE}$ is the base of this fibration). We already encountered examples of Kähler polarizations, while an example of the second kind is provided by the \emph{Schrödinger polarization} $\Pm_\mathrm{Sch}$; in the scheme of diagram (\ref{diag_cored}) gathering the relevant information, these appear as
   \begin{equation}\label{diag:P}
     \Pm_g:
     \begin{tikzcd}[column sep=small]
       & T^\ast K \ar[ld, equals] \ar[rd] & \\
       (T^\ast K,J_g) \ar[rr] & & \{ pt \}
     \end{tikzcd}
     \qquad
     \Pm_\mathrm{Sch}:
     \begin{tikzcd}[column sep=small]
       & T^\ast K \ar[ld, "\pi" swap] \ar[rd, "\pi"] & \\
       K \ar[rr, equals] & & K
     \end{tikzcd}
   \end{equation}
   Note that all of these polarizations are invariant under the $K\times K$-action. 
 \end{remark}
 
 Returning to the general setting, recall that the canonical bundle $\KK_\Pm$ of a polarization $\Pm$ (see Section 10.2 of \cite{woodhouse:1991}) is defined as the line bundle whose sections are the  $n$-forms that are annihilated by $\overline{\Pm}$.
 The tensor $J_{\DD}$ equips the isotropic leaf space $B_{\DD}$ with a canonical bundle relative to the application $\phi$ which we denote by $\KK_\phi$ in the following way: the $+i$-eigenspaces of $J_{\DD}$ form a subvector bundle
 \[
  (\ker d\phi)^{1,0}_{J_{\DD}} \subset (\ker d\phi)\otimes \CC
 \]
of the complexification of the vertical bundle of $\phi$, and its determinant is the \emph{relative canonical bundle}
\begin{equation}\label{def_relcan}
  \KK_\phi := \bigwedge^\mathrm{top} (\ker d\phi)^{1,0}_{J_{\DD}} .
\end{equation}
It is clear from this definition that when restricted to any fiber $\phi^{-1}(b)$ we recover the usual holomorphic canonical bundle.

 \begin{proposition}\label{prop_K_mixed}
   The canonical bundle $\KK_\Pm$ of a fibering polarization $\Pm$ with respect to (\ref{diag_cored}) satisfies
   \[
   \KK_{\Pm} = \pi_{\DD}^\ast \KK_{\phi} \otimes \pi_{\EE}^{\ast} \det T^\ast_{\CC} B_{\EE}
   = \pi_{\DD}^\ast \left( \KK_{\phi} \otimes \phi^{\ast} \det T^\ast_{\CC} B_{\EE} \right) ,
   \]
   where $\KK_{\phi}$ is the relative canonical bundles of $\phi$ and $T^\ast_{\CC} B_{\EE}$ is the complexification of the cotangent bundle of $B_{\EE}$.  
 \end{proposition} 

 To keep notations a bit lighter, we will denote the canonical bundle of the polarizations $\Pm_g$, $\Pm_\mathrm{Sch}$, and $\Pm_\mathrm{KW}$ (to be defined below) by $\KK_g, \KK_\mathrm{Sch}$, and $\KK_\mathrm{KW}$, respectively.

A crucial role is played by the polarized sections of the canonical bundle of a polarization; in the present setting, there is a natural definition which combines and generalizes the Kähler and real case (for these, cf. the detailed discussion in \cite{hall:2002} besides the standard reference \cite{woodhouse:1991}).

\begin{definition}\label{dfn_polsec}
  For a fibering polarization, the sheaf of \emph{polarized sections} of $\KK_\Pm$ consists of those $\Omega$ which locally on the open sets $U$ of a covering can be written as
  \[
   \Omega\vert_U = \pi_\DD^\ast \Omega_\phi \otimes \pi_\EE^\ast \Omega_\EE
  \]
  where $\Omega_\phi$ is a smooth section of $\KK_\phi$ which is fiberwise holomorphic, and $\Omega_\EE$ is a real top differential form on the coisotropic leaf space $B_\EE$.
\end{definition}

\subsection{Half-forms and the definition of the quantum spaces}\label{sss_defquan} To define the quantization associated to a polarization, we need to discuss half-forms.
For a fixed polarization $\Pm$, a \emph{half-canonical bundle} $\kappa_\Pm \to M$ is a complex line bundle together with a fixed isomorphism $\kappa_\Pm^{\otimes 2} \cong \KK_\Pm$. It inherits a notion of polarized sections -- a section $s$ is polarized if its square $s^{\otimes 2}$ is a polarized section of $\KK_\Pm$.

The quantization associated to a polarization is defined in terms of two requirements, a local polarization condition, and a global ``finite energy'' condition. Since mixed polarizations with  some compact real directions force concentration of \emph{quantum states} (or polarized sections obeying an appropriate finite energy condition) on subvarieties, we need an appropriate notion of distributional polarized sections.

Consider therefore the following two subsheaves
\begin{align*}
  \HH_\Pm^\infty: \qquad &  U \mapsto C^\infty_{\Pm}(U,L) \underset{C^\infty_{\Pm}(U,\CC)}{\otimes} C^\infty_{\Pm}(U,\kappa_\Pm) \subset C^\infty(U,L \otimes \kappa_\Pm) , \\
  \HH_\Pm^{-\infty}: \qquad &  U \mapsto C^{-\infty}_{\Pm}(U,L) \underset{C^\infty_{\Pm}(U,\CC)}{\otimes} C^\infty_{\Pm}(U,\kappa_\Pm) \subset C^{-\infty}(U,L \otimes \kappa_\Pm) ,
\end{align*}
of smooth respectively distributional sections of $L\otimes \kappa_\Pm$ which can \emph{locally} be written as products of $\Pm$-polarized sections (as indicated by the subscripts). Woodhouse \cite[\S 9.3]{woodhouse:1991} calls sections satisfying the smooth variant of this condition \emph{$\Pm$ wave functions}.

Importantly, this extension is conservative for Kähler polarizations:
\begin{proposition}[See {\cite[Lemma 2]{kajiwara.yoshida:1968}}]
A distribution which satisfies the Cauchy--Riemann equations on an open subset of $\CC^n$ (or of some complex manifold) is a holomorphic function.
\end{proposition}

\begin{remark} \item[(i)] The reason for writing the definition down in this form is that it is evident how to define polarized distributional sections on $L$, where this condition is given by a system of differential equations on smooth sections, which generalize directly to a system of differential equations on distributional sections. On the half-form part, on the other hand, a polarization is determined by taking the square of a local section, which cannot in general be done for a distributional section.

\item[(ii)] It is important to note that this product condition defines only a pre-sheaf, and that in general it is \emph{not} possible to write a $\Pm$-polarized section \emph{globally} as a product of polarized sections of $L$ and $\kappa_\Pm$.
\end{remark}

\subsection{Local structure of fibering polarizations and the Bohr--Som\-mer\-feld condition} The structure of fibering polarizations actually admits a ``normal form'' locally on $B_{\EE}$ in general:

\begin{proposition}\label{prop_mp_props} Assume $\Pm$ is a fibering polarization such that the maps $\pi_{\DD}, \pi_{\EE}, \phi$ in diagram (\ref{diag_cored}) are proper with connected fibers. Then, over the open subset $B_{\EE}^{\rm reg}$ of regular values of $\pi_{\EE}$, the following properties hold:
\item[(a)] The fibers of $\pi_{\DD}$ are isotropic tori.
\item[(b)] Over sufficiently small open subsets $U \subseteq B_{\EE}^{\rm reg}$, these isotropic tori are the orbits of a Hamiltonian torus action defined on $\pi_{\EE}^{-1}(U)$.
\item[(c)] The base $B_{\EE}^{\rm reg}$ carries an integral affine structure.
\end{proposition}
\begin{proof}
The functions in Nirenberg's Theorem \ref{thm_nir} come from $k$ (real) coordinates on $B_{\EE}$
\[
 f_1,\dots,f_k : U \to \RR, \qquad U \underset{\text open}{\subset} B_{\EE}
\]
and $n-k$ complex functions on $B_{\DD}$ which are local $J_{\DD}$-holomorphic coordinates on the fibers of $\phi$,
\[
 z_1, \dots, z_{n-k}: V \to \CC, \qquad V \underset{\text open}{\subset} B_{\DD}
\]
such that the pull-backs $f_i \circ \pi_{\EE}, z_j \circ \pi_{\DD}$ of these functions Poisson-commute. In particular, at any point $p\in \pi_{\EE}^{-1}U$ the flow along the Hamiltonian vector fields $X_{f_i \circ \pi_{\EE}}$ generates a map $\RR^k \to \pi_{\DD}^{-1}(\pi_{\DD}(p))$ which is an isotropic immersion onto the fiber of $\pi_{\DD}$ through $p$. As by hypothesis these fibers are compact, they have therefore to be tori which are the orbits of a (locally defined) Hamiltonian torus action. As in the construction of action-angle coordinates, we conclude that there exists an integral affine structure on $B_{\EE}^{\rm reg}$ (defined by those functions on $U$ whose Hamiltonian flow at time 1 is actually periodic). 
\end{proof}

\begin{remark} An equivalent formulation of Proposition \ref{prop_mp_props} is that for the purpose of statements local on the coisotropic leaf space, we can without loss of generality assume the fibrations (\ref{diag_cored}) involved to be generated by a Hamiltonian torus action
\begin{equation}\label{diag_stdpol}
    \begin{tikzcd}[column sep=small]
       & \widetilde{U} := \pi_{\EE}^{-1}(U) \ar[ld] \ar[rd, "\pi_{\EE}=\mu_T"] & \\
       \widetilde{U}/T \ar[rr, "\phi"] & & U \subset \tt^\ast
     \end{tikzcd}
\end{equation}
As will be evident from the definition of the Kirwin--Wu polarization below, both the Hamiltonian action which generates the isotropic tori and the affine structure on the base of the coisotropic fibration are of course global -- they are the singular torus action generated by the invariant moment map $\mu_\mathrm{inv}:(T^\ast K)_\mathrm{reg} \to \breve{\tt}^\ast_+$, and the natural affine structure on the positive Weyl chamber.
\end{remark}

It is an immediate consequence of Proposition \ref{prop_K_mixed} that the canonical bundle $\KK_{\Pm}$ of a polarization fibering with respect to (\ref{diag_stdpol}) comes with a canonical lift of the Hamiltonian torus action; therefore also any half-form bundle $\kappa_{\Pm}$ comes with a canonical lift of the action of the corresponding Lie algebra $\tt$ (in general, monodromy can obstruct the action of the torus itself to lift there). Also, if $(L,\nabla)$ is a prequantum line bundle, there is a canonical lift of the Hamiltonian $\tt$-action (cf. for example \cite{mundet:2001}).

\begin{definition}\label{defBS} Assume $\Pm$ is a fibering polarization with respect to (\ref{diag_stdpol}); a point $\xi\in U$, a fiber $\phi^{-1}(\xi) \subset B_\DD$, or a coisotropic leaf $\mu_T^{-1}(\xi)\subset M$, respectively, is \emph{Bohr--Som\-mer\-feld} if the action of the Lie algebra
$$
 \tt \circlearrowright L\otimes\kappa_{\Pm}\vert_{\mu_T^{-1}(\xi)}   
$$
descends to $T$ when the line bundle is restricted to $\mu_T^{-1}(\xi)$.
\end{definition}

\begin{theorem}\label{thm_BS1}
Suppose $\Pm$ is a polarization on (\ref{diag_stdpol}), and $U$ is an open set containing no Bohr--Som\-mer\-feld point; then any $\Pm$-polarized section of $L\otimes\kappa_{\Pm}$ on $\widetilde{U}$ is identically 0.
\end{theorem}
\begin{proof} This argument is identical to the one in \cite[Prop.~3.1.(i)]{baier.florentino.mourao.nunes:2011}; however, we sketch its proof for completeness.

Since $\Pm$ is $T$-invariant, the Lie algebra action of $\tt$ on $L\otimes\kappa_{\Pm}$ induces an action of $\tt$ on the space of polarized sections defined on any open subset which is a union of $T$-orbits. The monodromy of this action is a map $m$ which factors through the moment map of the torus action
\begin{equation}\label{diag_monodromy}
\begin{tikzcd}[column sep=small]
  \tt_\ZZ \times \widetilde{U} \ar[rr, "m"] \ar[rd, "\id \times \mu_T" swap] & & U(1) \\
  & \tt_\ZZ \times U \ar[ru]
\end{tikzcd}
\end{equation}
and a point $p$ lies in a Bohr--Som\-mer\-feld leaf if and only if the monodromy at $p$ is trivial,
\[
m\vert_{\tt_\ZZ\times\{p\}} \equiv 0 .
\]
Therefore, if $p$ does not lie in a Bohr--Som\-mer\-feld leaf, there is a nontrivial element $\eta\in\tt_\ZZ$ such that $m(\eta,p)\neq 1$, and the same holds for all points $p'\in\widetilde{U}'$ for a sufficiently small neighborhood $U'$ of $\mu_T(p)$ in $U$. Acting on any polarized section $s\in \HH^{-\infty}_{\Pm}(\widetilde{U}')$, we obtain that
\[
 s = m(\eta,\cdot )s ,
\]
so any such $s$ has to vanish.
\end{proof}

\begin{remark}
  The moment map $\mu_T$ of the torus action in (\ref{diag_stdpol}) is only unique up to translation; by the Duistermaat--Heckman theorem \cite{duistermaat.heckman:1982}, it can be chosen so that the monodromy function $m$ in (\ref{diag_monodromy}) takes the form
  \[
   m(\eta,p) = e^{2\pi i \langle \mu_T(p),\eta \rangle} .
  \]
  This reduces the indeterminacy in $\mu_T$ to a translation in $\tt^\ast_\ZZ$. 
\end{remark}

We expect that the distributional quantum states supported on a Bohr--Som\-mer\-feld leaf are in natural bijection with the holomorphic sections of the corresponding reduction:
\begin{conjecture}\label{thm_BS2}
The subspace $\Qm_\xi \subset \HH^{-\infty}_{\Pm}(\widetilde{U})$ of polarized sections which satisfy a certain finite energy condition supported at a Bohr--Som\-mer\-feld point $\xi$ is isomorphic to the space of holomorphic sections of the reduction of $L\otimes\kappa_{\Pm}$ to $\phi^{-1}(\xi)$,
\[
 \Qm_\xi \cong H^0(\phi^{-1}(\xi),L\otimes\kappa_{\Pm}\vert_{\mu_T^{-1}(\xi)}/T) .
\]
\end{conjecture}
We come back to discussing this issue for the specific case of the Kirwin--Wu polarization below.

\subsection{Fibering polarizations associated to Hamiltonian group actions}\label{sec_Fourierpol} Even though we are mainly interested in the specific case of $T^\ast K$ here, we pause for a moment to illustrate the applicability of the notion of fibering polarizations in a wider context.

Namely, consider a Hamiltonian $G$-manifold $(M,\omega)$ with moment map $\mu: M\to \gg^\ast$, and assume that $G$ is compact and $M$ is ``convex'' in the sense of \cite{knop:2011}. The singular dual pair supporting the fibering polarizations we consider is given by the diagram
\begin{equation} \label{diag_invpol}
  \begin{tikzcd}
    & M \ar[ld, "" swap] \ar[rd, "\mu_{\textrm{inv}}"] & \\
    M / \sim \ar[rr, "\phi"] & & P
  \end{tikzcd}
\end{equation}
where the Kirwan polytope $P$ is equipped with the algebra of smooth functions $C^\infty(P/W_M)$ described in \cite[Cor.~3.6]{knop:1997}, providing an isomorphism
\[
\mu_{\textrm{inv}}^\ast C^\infty(P/W_M) = Z(C^\infty(M)^G,\{\ ,\ \})
\]
with the center of the Poisson algebra of $G$-invariant functions on $M$, and the quotient $M \to M/\sim$ in the left-hand leg is by the foliation spanned by the Hamiltonian vector fields of these functions.

Polarizations based on this singular dual pair have one ``evident'' property which is desirable from the point of view of representation theory:
\begin{lemma}\label{lemma_casimirs}
If $\Pm$ is a fibering polarization with respect to Diagram (\ref{diag_invpol}), the (generalized) Casimirs of the group $G$ are $\Pm$-polarized.
\end{lemma}
\begin{proof}
Follows immediately from the inclusion
\[
 \mu_{\textrm{inv}}^\ast \left( C^\infty(\tt^\ast)^W \right) =
 \mu^\ast \left( C^\infty(\gg^\ast)^G \right) \subset
 Z(C^\infty(M)^G,\{\ ,\ \}) .
\]
\end{proof}

If the action is multiplicity-free, the dual pair (\ref{diag_invpol}) simplifies and becomes
\begin{equation} \label{diag_Fourierpol}
  \begin{tikzcd}
    & M \ar[ld, "\mu" swap] \ar[rd, "\mu_{\textrm{inv}}"] & \\
    (\mu(M),\{\ ,\ \}_{\gg^\ast}) \ar[rr, "\phi"] & & (P,0)
  \end{tikzcd}
\end{equation}
where the differentiable space structure on $\mu(M)$ has been studied in detail in \cite{knop:1997}. As the fibers of $\phi$ are coadjoint orbits of $G$, they carry a finite number of $G$-invariant complex structures, exactly one of which is such that the Kostant--Kirillov--Souriau symplectic form of it (that is, the restriction of $\{\ ,\ \}_{\gg^\ast}$) defines an ample class.

\begin{definition}
The \emph{Fourier polarization} $\Pm_{\mathrm{Fourier}}$ is the fibering polarization defined by Diagram (\ref{diag_Fourierpol}) and the $\phi$-ample complex structure over the regular values of $\mu_{\mathrm{inv}}$.
\end{definition}

\begin{remark}
\item[(i)] Given that we are in the multiplicity-free case, the left leg $\mu$ of Diagram (\ref{diag_Fourierpol}) coincides with the ``optimal moment map $\mathfrak{J}$'' and the right leg $\mu_{\mathrm{inv}}$ with the map $\pi$ of \cite{ortega:2003}.

\item[(ii)] If the action is not multiplicity-free, one can still attempt to define a singular dual pair starting from the invariant moment map; in any case, the fibers of $\phi$ will no longer be coadjoint orbits, and existence and uniqueness of $G$-invariant complex structures are not clear a priori. It is our expectation that degenerations along Mabuchi rays generated by convex functions on $P$ will \emph{define} such complex structures (and therefore generalizations of the Fourier polarization) in the not multiplicity-free case.

\item[(iii)] For compact $G$, the invariant moment map integrates to the action of a compact torus over an open dense subset of $M$ in general, cf. \cite[Lemma 2.3]{knop:2011}, and the Bohr--Sommerfeld points in $P$ will differ by elements in the character lattice of this torus.

\item[(iv)] If we drop the compactness requirement of $G$,  we end up having both compact and non-compact directions in the group which integrates the invariant moment map over the regular set, and the Bohr--Sommerfeld points will lie in intersections of $P$ with linear subspaces; see \cite{dazord.delzant:1987} and references therein for a discussion of the relevant generalization of the Duistermaat--Heckman theorem.
\end{remark}

As a corollary to Theorem \ref{thm_BS1} we obtain:
\begin{corollary}
  Suppose $G$ is compact, $\gg$ has trivial center, and $M$ is a multiplicity-free Hamiltonian $G$-manifold whose moment map hits the interior of the positive Weyl chamber. Then any distributional $\Pm_{\mathrm{Fourier}}$-polarized section is supported on the inverse image of the intersection of the (relative interior) of the Kirwan polytope with the $\rho$-translate of the character lattice of the maximal torus.
\end{corollary}

We give an outline of the conjectures relating these fibering polarizations to Mabuchi space in more general cases than $T^\ast K$ below in Section \ref{sec_conjectures}.

\section{The Kirwin--Wu polarization $\Pm_\mathrm{KW}$ on $T^*K$}\label{sect_KW}

The Kirwin--Wu polarization is the Fourier polarization of $T^\ast K$ as Hamiltonian $G=K\times K$-manifold, or more explicitly:

\begin{definition} The \emph{Kirwin--Wu polarization} $\Pm_\mathrm{KW}$ is the mixed polarization on the dense open subset $(T^\ast K)_\mathrm{reg}$ defined by applying Proposition \ref{prop_polarizations} to the (co-)isotropic fibrations
\begin{equation}\label{KW_diag}
  \begin{tikzcd}[column sep=small]
    & (T^\ast K)_\mathrm{reg} \ar[ld, "\mu" swap] \ar[rd, "\mu_\mathrm{inv}"] & \\
    \kk^\ast \times_{\breve{\tt}^\ast_+} \kk^\ast \ar[rr, "\phi"] & & \breve{\tt}^\ast_+
  \end{tikzcd}
\end{equation}
and the isomorphism
\begin{equation}\label{KW_cxstr}
\begin{tikzcd}[row sep=0pt, column sep=tiny]
\breve{\tt}^\ast_+ \times K_{\CC}/B \times K_{\CC}/B_{-} \ar[r, equals] & \breve{\tt}^\ast_+ \times K/T \times K/T \ar[r, "\cong"] & \kk^\ast \times_{\breve{\tt}^\ast_+} \kk^\ast \\
 & (\xi_+, k_1 T, k_2 T) \ar[r, mapsto] & ( \Ad_{k_1}^\ast \xi_+ , -\Ad_{k_2}^\ast \xi_+) 
 \end{tikzcd}
\end{equation}
which equips the fibers of the map $\phi$ in (\ref{KW_diag}) with $K\times K$-invariant complex structures (cf. Section \ref{prelimLie} for notations regarding the Borel subgroups $B,B_{-}$).
\end{definition}

\subsection{Local description of the Kirwin--Wu polarization} Local $\Pm_\mathrm{KW}$-polarized functions (besides the evident $f\circ \mu_\mathrm{inv}$ for arbitrary real functions $f:\breve{\tt}^\ast_+ \to \RR$) can  be described using the complex geometry of the coadjoint orbits which appear in its definition. To this end we slightly adapt the well-known description of invariant complex structures on $K/T$ discussed in Section \ref{prelimLie}.

\subsubsection{The complex directions of Kirwin--Wu} \label{sec:cplx Kirwin-Wu} Since we are dealing with a $K\times K$-action, we need a version of this which embeds two copies of the coadjoint orbit with the correct equivariance; this is achieved by using $\End V_\lambda$ instead of $V_\lambda$. For any weight $\nu\in \tt^\ast_\ZZ$ and the corresponding character $\chi_\nu:T\to U(1)$, consider the orthogonal projection $P_\nu \in \End V_\lambda$ onto the weight eigenspace
\begin{equation}\label{eq_weightproj}
P_\nu v = \int_{T} \chi_\nu(t^{-1}) \pi_{\lambda}(t)v\, dt , \qquad \forall v\in V_\lambda .
\end{equation}
Using the $KAK$-decomposition (\ref{KAK_deco}), for any endomorphism $A\in \End V_\lambda$ introduce the function
\begin{equation}\label{eqn_FlambdaA}
 (x_1,\xi_+,x_2) \mapsto \pi_\lambda(x_1) P_{\lambda} \pi_\lambda(x_2^{-1}) A  \in \End V_\lambda .
\end{equation}
These functions descend to well-defined functions $\tilde{F}_{\lambda,A}:(T^\ast K)_\mathrm{reg} \to \End V_\lambda$ on the regular stratum which are equivariant with respect to the singular torus action (which corresponds to $(x_1,\xi_+,x_2)\star t=(x_1 t^{-1},\xi_+,x_2)$ on the $KAK$-decomposition),
\[
\tilde{F}_{\lambda,A}\big( (x,\xi)\star t \big) = \chi_{\lambda}(t^{-1}) \tilde{F}_{\lambda,A}(x,\xi) .
\]
As before, the quotients of the traces $F_{\lambda,A}(x,\xi) := \tr \tilde{F}_{\lambda,A}(x,\xi)$
\[
(x,\xi) \mapsto \frac{F_{\lambda,A_1}(x,\xi)}{F_{\lambda,A_2}(x,\xi)}
\]
are therefore $T_\mathrm{inv}$-invariant functions on an open subset of the regular stratum which extend analytically to the partial complexification $K_{\CC} \times \breve{\tt}^\ast_+ \times K_{\CC}$ (by the same expression), where they factor through the quotient
\[
 K_{\CC} \times \breve{\tt}^\ast_+ \times K_{\CC} \to K_{\CC}/B \times \breve{\tt}^\ast_+ \times K_{\CC}/B_{-} .
\]
To see this, observe that the maximal torus $T\times T \subset K\times K$ acts via the character $\chi_\lambda \times \chi_{-\lambda}$ on the point $P_\lambda \in \End V_\lambda$, hence the stabilizer of the corresponding point in the projectivization $[P_\lambda] \in \PP \End V_\lambda$ has stabilizer the Borel subgroup $B \times B_{-}$ which corresponds to the Weyl chamber (of $K\times K$) containing $(\lambda,-\lambda)$.
We have proved:

\begin{proposition}\label{KW_nirenberg} Consider a highest weight representation $V_\lambda$ of $K$ and two endomorphisms $A_i \in \End V_\lambda$. Then the complex function
  \[
    (x,\xi) \mapsto \frac{F_{\lambda,A_1}(x,\xi)}{F_{\lambda,A_2}(x,\xi)}
  \]
  (which is well defined on the open subset $U \subset (T^\ast K)_\mathrm{reg}$ on which its denominator does not vanish) is $T_\mathrm{inv}$-invariant and hence descends to a function $z_{\lambda,A_1 , A_2}: U' = U/T_\mathrm{inv} \to \CC$ on an open subset of $B_{\DD} \cong \kk^\ast \times_{\breve{\tt}^\ast_+} \kk^\ast$.  

  The union over all functions of this kind, or equivalently the union over appropriate choices of $A_i$'s for a single $\lambda$ so that the corresponding line bundle is very ample on the coadjoint orbit, generate the complex directions of the Kirwin--Wu polarization.
\end{proposition}

\begin{remark}\label{products}
Note that for endomorphisms given by simple tensors $A= u\otimes v^*, u\in V_\lambda, v^*\in V_\lambda^*$, $F_{\lambda,A}$ can be written as a product
\[
F_{\lambda, A} (x,\xi) = \tr (\pi_\lambda(x_1) v_\lambda\otimes v^*) \tr (\pi_\lambda(x_2^{-1}) u\otimes v_\lambda^*) .
\]
\end{remark}

\begin{remark}
Note that although the $\xi$-component does not show up explicitly in the expressions (\ref{eqn_FlambdaA}), the functions are of course \emph{not} independent of $\xi$, since the factorization $x=x_1 x_2^{-1}$ depends crucially on it.
\end{remark}

\subsection{Fourier decomposition with respect to $T_\mathrm{inv}$ and relation to the Borel--Weil Theorem}
\subsubsection{Fourier harmonics} For any function $f:(T^\ast K)_\mathrm{reg} \to \CC$ on the regular stratum (which for ease of exposition we take to be complex-valued, although the same procedure applies verbatim for values in other vector spaces), we define a family of functions $\widehat{f}_\nu: (T^\ast K)_\mathrm{reg} \to \CC$ indexed by the characters $\chi_\nu, \nu \in \tt^\ast_{\ZZ}$, of $T_\mathrm{inv}$, by setting
\[
 \widehat{f}_\nu(p) := \int_{T_\mathrm{inv}} \! \chi_\nu(t)  f(p \star t)  dt .
\]
These Fourier harmonics of $f$ transform via the (inverse of the) corresponding character of the singular torus,
\[
\widehat{f}_\nu( p \star t ) = \chi_\nu(t^{-1}) \widehat{f}_{\nu}(p) ,
\]
and in an appropriate sense $f$ will be the sum of its Fourier series,
\[
 f(p) = \sum_{\nu} \widehat{f}_\nu (p) .
\]
Since we will only be concerned with functions with finitely many non-vanishing summands, we do not need to make the analytical aspects of this statement more precise.

The Fourier harmonics can be made explicit for the $J_g$-holomorphic matrix coefficients
\[
f^g_{\lambda,A}(x,\xi) = \tr (\pi_\lambda(xe^{i d_\xi g})A)
\]
of the Kähler polarization of Proposition \ref{prop_Pg}. Consider the orthogonal projections onto the weight eigenspaces $P_\nu \in \End V_\lambda$ of (\ref{eq_weightproj}) and, for any point $p=(x=x_1 x_2^{-1},\xi=\Ad_{x_2}^\ast \xi_+)$ consider the conjugate
\[
 P_\nu(x,\xi) := \pi_{\lambda}(x_2) P_{\nu} \pi_{\lambda}(x_2^{-1})
\]
of $P_{\nu}$, which is independent of the indeterminacy $x_2 \leadsto x_2 t$, hence well-defined.

\begin{proposition}\label{fourier-1}
The Fourier harmonics of the matrix coefficient $f^g_{\lambda,A}$ are
  \begin{equation}\label{deco-1}
   \widehat{(f^g_{\lambda,A})}_\nu(x,\xi) = \tr \big( \pi_{\lambda}(x e^{i d_\xi g})P_\nu(x,\xi)A \big) ,
  \end{equation}
  so in particular they vanish unless $\nu$ is in the intersection $\overline{W \lambda}$ of the weight lattice with the convex hull of the Weyl-group orbit of the highest weight.
\end{proposition}
\[
\vcenter{\begin{tikzpicture}[scale=0.66]
  \pgfmathsetmacro\bx{cos(60)}
  \pgfmathsetmacro\by{sin(60)}
  \foreach \k in {1,...,6} {
  \draw[dashed] (0,0) -- +(\k * 60:3.5);
  }
  \fill[gray!25] (0,0) -- +(60:3) -- (3,0) -- cycle;
  \draw (0.75*2+0.75*\bx,0.75*\by) -- (-0.75*\bx,0.75*3*\by) -- (-2*0.75,0.75*2*\by) -- (-2*0.75,-0.75*2*\by) -- (-0.75*\bx,-0.75*3*\by) -- (0.75*2+0.75*\bx,-0.75*\by) -- cycle;
  \foreach \m in {-4,...,4} {
  \foreach \n in {-4,...,4} {
    \ifnum \numexpr \n+\m < 5
     \ifnum \numexpr \n+\m > -5
      \node[circle,inner sep=0pt,minimum size=5pt,draw=black] at (0.75*\n+0.75*\m*\bx,0.75*\m*\by) {};
     \fi
    \fi
  } }
  \node[right] at (0.75*2+0.75*\bx,0.75*\by) {$\lambda$};
 \end{tikzpicture}}
 \hspace{-7cm} \quad f^g_{\lambda,A} = \sum_{\nu \in \overline{W\lambda}} \widehat{(f^g_{\lambda,A})}_\nu
\]
\begin{proof} Given that the projection operators $P_\nu(x,\xi)$ are common conjugates of the weight projectors $P_\nu$ at each point $(x,\xi)$, they sum to the identity
  \[
   \id_{V_\lambda} = \sum_{\nu \in \overline{W\lambda}} P_\nu(x,\xi) .
  \]
  It follows that
  \[
   f^g_{\lambda,A}(x,\xi) = \sum_{\nu \in \overline{W\lambda}} \widehat{(f^g_{\lambda,A})}_\nu(x,\xi) ,
  \]
  so to conclude the proof we only need to verify the appropriate equivariance of each summand. But this follows directly,
\begin{align*}
    \widehat{(f^g_{\lambda,A})}_\nu\big( (x_1 x_2^{-1},&\Ad_{x_2}^\ast \xi_+) \star t \big)  = \widehat{(f^g_{\lambda,A})}_\nu (x_1 t^{-1} x_2^{-1},\Ad_{x_2}^\ast \xi_+) \\
    & = \tr \big( \pi_{\lambda}(x_1 t^{-1} x_2^{-1} e^{i d_{\Ad_{x_2}^\ast\xi_+} g})P_\nu(x_1 t^{-1} x_2^{-1},\xi)A \big) \\
    & = \tr \big( \pi_{\lambda}(x_1 t^{-1} x_2^{-1} x_2 e^{i d_{\xi_+} g} x_2^{-1}) \pi_\lambda(x_2)P_\nu \pi_\lambda(x_2^{-1})A \big) \\
    & = \tr \big( \pi_{\lambda}(x_1 e^{i d_{\xi_+} g} t^{-1}) P_\nu \pi_\lambda(x_2^{-1})A \big) \\
    & = \chi_\nu(t^{-1}) \tr \big( \pi_{\lambda}(x_1 e^{i d_{\xi_+} g}) P_\nu \pi_\lambda(x_2^{-1})A \big),
  \end{align*}
where besides elementary properties we only used equivariance of the Legendre transform $\xi \mapsto d_\xi g$ from Proposition \ref{prop_Pg}.
\end{proof}

\subsection{Convergence of polarizations} The Fourier decomposition (\ref{deco-1}) relates to the behavior of the Kähler structures $I_{g_t}$ along Mabuchi geodesic rays: different Fourier harmonics will scale with different rates with regard to the geodesic time $t$. This becomes apparent by factoring out the dependence on the complex structure in (\ref{deco-1}):

\begin{lemma}\label{fourier-2}
The Fourier harmonic $\widehat{(f^g_{\lambda,A})}_\nu$ of the matrix coefficient $f^g_{\lambda,A}$ can be written as
\begin{equation*}\label{deco-2}
  \widehat{(f^g_{\lambda,A})}_\nu\left( x=x_1 x_2^{-1},\xi = \Ad_{x_2}^\ast \xi_+ \right)
    = e^{\langle \nu , d_{\xi_+} g \rangle}\tr \big( \pi_{\lambda}(x_1)P_\nu \pi_\lambda(x_2^{-1}) A \big) .
  \end{equation*}
\end{lemma}
\begin{proof} The statement follows from the calculation in the proof of Proposition \ref{fourier-1} and the observation that the value of the (analytic continuation of the) character $t \mapsto \chi_\nu(t^{-1})$ on the element $e^{i d_{\xi_+}g} \in T_{\CC}$ equals $e^{\langle \nu, d_{\xi_+}g \rangle}$ -- here we use that $\LL_g(\tt^*)\subseteq \tt$, cf. Proposition~\ref{prop_Pg}.
\end{proof}

Taking into account the following Lemma~\ref{lemma_posLegendre}, it follows that the dominant contribution (as $t\to \infty$) comes from the Fourier harmonic corresponding to the highest weight $\lambda$.

\begin{lemma}\label{lemma_posLegendre} Consider $g:\kk^\ast\to \RR$ strictly convex and invariant under the coadjoint action, and fix a positive weight $\lambda$; for any other weight $\nu \in \overline{W\lambda}$ and any element $\xi_+ \in \breve{\tt}^\ast_+$ in the interior of the positive Weyl chamber,
\[
 \langle \lambda-\nu, \LL_g(\xi_+) \rangle > 0 .
\]
\end{lemma}
\begin{proof} Since the Legendre transform $\LL_g : \tt^\ast \to \tt$ is a Weyl-group equivariant diffeomorphism, it maps the walls of the Weyl chambers (which are the elements with nontrivial stabilizers in $W$) to each other, hence it also preserves the open positive Weyl chamber. Since the difference $\lambda-\nu$ is a non-negative sum of simple roots, and any $\xi_+\in\breve{\tt}^\ast_+$ is a positive linear combination of fundamental weights, the result follows.
\end{proof}

\begin{lemma}\label{lemma_flambda_conv}
  For a Mabuchi geodesic ray of Kähler polarizations specified by a family of convex functions $t \mapsto g_t = g+t h: \kk^\ast \to \RR$ with $h$ uniformly convex,
  \[
    \lim_{t \to \infty} e^{-\langle \lambda , \LL_{g_t}\circ \mu_\mathrm{inv}\rangle }f_{\lambda,A}^{g_t} = F_{\lambda,A}
  \]
  pointwise on the regular stratum.  
\end{lemma}

\begin{proof} Using the Fourier decomposition from Lemma~\ref{fourier-2} we have
  \begin{align*}
    \left( e^{-\langle \lambda , \LL_{g_t}\circ \mu_\mathrm{inv}\rangle}f_{\lambda,A}^{g_t}\right) (x,\xi) =
    \sum_{\nu}e^{\langle \nu-\lambda , d_{\xi_+} g_t \rangle}\tr \big( \pi_{\lambda}(x_1)P_\nu \pi_\lambda(x_2^{-1}) A \big)
  \end{align*}
with $\nu \in \overline{W\lambda}$, so the claim follows from Lemma~\ref{lemma_posLegendre}.
\end{proof}

Convergence of polarizations $\Pm_{g+th}\to\Pm_\mathrm{KW}$ as in Section \ref{par_families} will now follow once we find sufficiently many local $J_{g+th}$-holomorphic functions which converge to Kirwin--Wu polarized functions. From Lemma~\ref{lemma_flambda_conv}, we see that
\begin{equation}\label{KWconv1}
  \lim_{t \to \infty} \frac{f_{\lambda,A}^{g_t}}{f_{\lambda,B}^{g_t}}  = \frac{F_{\lambda,A}}{F_{\lambda,B}}
\end{equation}
pointwise away from the zero locus of $F_{\lambda,B}$ in $(T^\ast K)_\mathrm{reg}$, which takes care of the complex directions of the limit. For the real directions, observe that since 
\[
 \ln f_{\lambda,A}^{g_t} = \ln\left(\sum_{\nu}e^{\langle \nu-\lambda,d_{\xi_+}g_t\rangle} \tr \big( \pi_{\lambda}(x_1)P_\nu \pi_\lambda(x_2^{-1}) A \big) \right) + \langle \lambda, d_{\xi_+}g_t\rangle
\]
from the preceding Lemma \ref{lemma_posLegendre} we find that
\begin{equation}\label{KWconv2}
 \lim_{t \to \infty} \frac{1}{t}\ln f_{\lambda,A}^{g_t} = \langle \lambda, \LL_{h} \circ \mu_\mathrm{inv} \rangle .
\end{equation}

\begin{theorem}\label{thm_convpol} Consider a family of convex invariant functions $g+t h:\kk^\ast \to \RR$ with $h$ strictly convex. Then the Kähler polarizations $\Pm_{g+th}$ converge pointwise to the Kirwin--Wu polarization as $t \to \infty$.
\end{theorem}

\begin{proof} It suffices to exhibit convergence of a sufficiently large number of $J_{g_t}$-polarized (local) functions to Kirwin--Wu polarized functions in the neighborhood of any point, which (taking into account the description of the latter in Proposition \ref{KW_nirenberg}) we just did in (\ref{KWconv1}) and (\ref{KWconv2}). Regarding the real directions, this uses the fact that we can find $\dim T$ linearly independent regular weights $\lambda \in \breve{\tt}^\ast \cap \tt^\ast_{\ZZ}$.
\end{proof}

\begin{remark}\label{rem:loc-unif-convergence} The convergence could be improved to uniform convergence on compact neighborhoods of points inside the regular stratum, but we will not make use of the analytical aspects of the convergence in what follows.
\end{remark}

\subsection{Triviality of the canonical bundle $\KK_\mathrm{KW}$ and half-forms in families} It is shown in Appendix~\ref{par_families} that convergence of the Kähler polarizations (which have trivializable canonical bundle) to the Kirwin--Wu polarization implies that also the canonical bundle of the latter is trivializable. This can also be seen directly from the definition as follows:

\begin{proposition} The canonical bundle of the Kirwin--Wu polarization is trivializable as $K\times K$-linearized line bundle.
\end{proposition}
\begin{proof}
  It is clear from the definition of the Kirwin--Wu polarization (\ref{KW_cxstr}), and the fact that the canonical bundle of the coadjoint orbit is $\KK_{K_\CC/B} = L_{-2\rho}$, that $\KK_\mathrm{KW}$ is isomorphic to $\pi_\mathrm{KW}^\ast (L_{-2\rho}\boxtimes L_{2\rho})$, where $\pi_{\mathrm{KW}}: (T^\ast K)_{\rm reg} \to K_\CC / B \times K_\CC / B_{-}$ is the composition of $\mu$ with the inverse of (\ref{KW_cxstr}) and the projection to the coadjoint orbits; the further pull-back of any line bundle $\pi^\ast_\mathrm{KW} (L_{\lambda_1}\boxtimes L_{\lambda_2})$ via the $KAK$-decomposition (\ref{KAK_deco}) is always trivial, and it is an immediate check that $\pi_\mathrm{KW}^\ast (L_{\lambda_1}\boxtimes L_{\lambda_2})$ is trivializable if and only if $\lambda_1 = -\lambda_2$.
\end{proof}

\begin{corollary}\label{cor_polcan}
\item[(i)] Any two $K\times K$-invariant sections of $\KK_{\mathrm{KW}}$ differ by the $\mu_{\mathrm{inv}}$-pullback of a function in $C^\infty(\breve{\tt}^\ast_+,\CC)$ on the subset where both are non-vanishing.

\item[(ii)]For any $K\times K$-invariant trivializing section $1_{\KK_{\mathrm{KW}}^{-1}}$, any polarized global section $s$ of $\KK_{\mathrm{KW}}^{-1}$ is of the form
\[
  s = (f\circ \pi_{\mathrm{KW}}) (g\circ\mu_\mathrm{inv}) 1_{\KK_{\mathrm{KW}}^{-1}} ,
\]
with $f \in H^0( K_{\CC}/B \times K_{\CC}/B_{-}, L_{2\rho}\boxtimes L_{-2\rho} )$ and  $g \in C^{\infty}(\breve{\tt}1^\ast_+,\RR)$.

\item[(iii)] For any $K\times K$-invariant trivializing section $1_{\KK_{\mathrm{KW}}}$ and $f \in H^0( K_{\CC}/B \times K_{\CC}/B_{-}, L_{2\rho}\boxtimes L_{-2\rho} )$, the section
\[
 \frac{1}{f\circ \pi_{\mathrm{KW}}} 1_{\KK_{\mathrm{KW}}} \in C^\infty_{\Pm_{\mathrm{KW}}}\left( \left\{ f\circ\pi_\mathrm{KW} \neq 0 \right\}, \KK_{\mathrm{KW}} \right)
\]
is polarized on the open subset where it is defined.
\end{corollary}

In order to construct the half-form correction over the entire family of polarizations under consideration, including Kähler, Schrödinger and Kirwin--Wu, we now construct a trivializing section of $\KK_{\bullet}$ which extends from the Kähler to the degenerate cases: start with a smooth $\mathrm{Ad}^\ast$-invariant symplectic potential $g$ as described in Section \ref{prelimKahler}. According to Equation $(3.5)$ in \cite{kirwin.mourao.nunes:2013} the holomorphic left $K_{\CC}$-invariant 1-forms ($j=1,\dots,n$)
 \begin{equation}
 \label{kmn3.5}    
     \Omega^j_{g} (x,\xi) := \sum_{k=1}^n \left[ e^{-i \mathrm{ad}_{d_\xi g}}\right]^j_k \omega^k +
     \left[\frac{1-e^{-i  \mathrm{ad}_{d_\xi g}}}{\mathrm{ad}_{d_\xi g}}\cdot \Hess_{g}(\xi)\right]^j_k d\xi^k,
\end{equation}
 on $(T^\ast K,\omega_\mathrm{std},I_{g})$ form a frame. Here the 1-forms $\{\omega^j\}_{j=1,\dots, n}$  are the pull-backs to $T^\ast K$ of the left-invariant 1-forms on $K$ associated to the choice of orthonormal basis in $\kk$, and the $\xi^k$ are the coordinates in the dual basis on $\kk^*$. Then a left-invariant holomorphic trivializing section of the canonical bundle $\KK_{(T^\ast K,J_{g})}$ is given by
\[
 \Omega_{g} := \bigwedge_{j=1}^n \Omega^j_{g}.
\]
It is necessary to modify these trivializations in order for them to extend to the mixed polarizations: 

\begin{proposition}\label{prop_trivialization}
  The map
  \begin{equation}
    g \mapsto \widehat{\Omega}_g := \frac{e^{-2 \langle \rho, \LL_g\circ \mu_\mathrm{inv} \rangle}}{1+\det \Hess_{\tt} g} \Omega_g \in C^\infty(M,\bigwedge^n T^\ast_{\CC}(T^\ast K))
  \end{equation}
  defines a trivializing section of the canonical bundle in the family of Kähler polarizations parameterized by $g$, which extends by continuity to a trivializing sections for the Schrödinger polarization on $T^\ast K$, and to a trivializing section for the Kirwin--Wu polarization on $(T^\ast K)_\mathrm{reg}$.
\end{proposition}

\begin{proof}
  For the Schrödinger polarization (which arises in the limit $g = th$ with $t \to 0$) this is evident from the expression (\ref{kmn3.5}), where only the left-invariant 1-forms survive; as the fraction in front of $\Omega_g$ converges to 1, we recover the pull-back of the Haar measure on $K$, as expected.

For the Kirwin--Wu polarization, the proof follows from the next lemma.
\end{proof}

\begin{lemma}\label{lem:Omega tilde}
Consider a Chevalley basis of $\kk\otimes {\mathbb C}$, 
 $\{H_j,E_{\alpha},E_{-\alpha}\}_{j=1,\dots, r, \alpha\in \Phi_+}$, corresponding to the choice of maximal torus $T\subset K$
 and positive Weyl chamber $\tt_+^*$,  where $\{H_j\}_{j=1,\dots, r}$ is a basis of $\tt$ and $\alpha_+$ denotes the set of positive roots, so that
 \[
 \mathrm{ad}_{H_j}E_{\pm\alpha}=\pm \sqrt{-1}\alpha (H_j) E_{\pm\alpha} .
 \]

 Here, $\alpha(H)>0$ for $\alpha\in \Phi_+$ and $H$ in the interior $\breve \tt_+$ of the dual positive Weyl chamber.
 Denote by $\omega^j, \omega^\alpha, d\xi^j, d\xi^\alpha$ the corresponding 1-forms on $T^\ast K$.
\item[(a)] The complex $n$-form $\widetilde{\Omega}_\infty$ on $(T^\ast K)_\mathrm{reg}$, which is $K\times K$-invariant and over $\{1\}\times \breve{\tt}_+^\ast$ equals
  \[
    \widetilde{\Omega}_\infty := \bigwedge_{j=1}^r  d\xi^j 
 \bigwedge_{\alpha\in \Phi_+} \left(\left(\langle\alpha,\xi_+\rangle_{\kk^*}\omega^{-\alpha} - d\xi^{-\alpha}\right)\wedge d\xi^{\alpha}\right) ,
  \]
  is a trivializing section of $\KK_{\mathrm{KW}}$.  
\item[(b)] Any $K\times K$-invariant smooth section of $\KK_{\mathrm{KW}}$ can be written as $(f\circ \mu_\mathrm{inv}) \widetilde{\Omega}_\infty$ for some function $f: \breve{\tt}^\ast_+ \to \RR$.
\item[(c)] Over the subset $\{1\}\times \breve{\tt}^\ast_+$,
  \[
    \lim_{t \to \infty} \widehat{\Omega}_{g+th} = i^r P(\xi_+)^{-1} \widetilde{\Omega}_\infty ,
  \]
  where
  \[
   P(\xi_+) := \prod_{\alpha \in \Phi_+} \langle\alpha,\xi_+\rangle_{\kk^*} .
   \]
   In particular, $\widehat{\Omega}_\infty$ does not vanish, hence is a $K\times K$-invariant trivializing section of $\KK_{\mathrm{KW}}$.
\end{lemma}

\begin{proof}
 From $K\times K$-invariance of $\Omega_{g+th}$ it is enough to prove that the limit exists for 
 $(x,\xi_+)$, where $\xi_+$ lies in the interior positive Weyl chamber (which is identified with the interior of the dual positive Weyl chamber, see Proposition \ref{p-1} and Lemma \ref{lemma_posLegendre}), so that the image of $\mathrm{ad}_{\xi_+}$ intersects $\tt$ only at $\{0\}$. 

 Proposition \ref{p-1} also implies that $\Hess_{g+th}(\xi_+)$ is block diagonal with respect to the decomposition $\kk\otimes \CC = (\tt \oplus \tt^\perp)\otimes \CC$ and that
 $$
 \Hess_{g+th}(\xi_+)_{\vert_{\tt^\perp}} = \mathrm{Diagonal} \left(\frac{\langle\alpha, d_{\xi_+} (g+th)\rangle}{\langle\alpha,\xi_+\rangle_{\kk^*}}\right).
 $$
 From (\ref{kmn3.5}), we then have
 $$
 \lim_{t\to +\infty} t^{-r} \bigwedge_{j=1}^r\Omega^j_{g+th} (x,\xi_+)=  i^r(\det(\Hess_{h})_{\vert_\tt}(\xi_+)) \bigwedge_{j=1}^r  d\xi^j,
 $$
 where the right-hand side is never vanishing.
 On the other hand, along the root spaces, we obtain, for $\alpha>0$,
 $$
 \Omega^{-\alpha}_{g+th} = e^{\langle\alpha, d_{\xi_+}(g+th)\rangle}\left(\omega^{-\alpha} -\langle \alpha,\xi_+\rangle_{\kk^*}^{-1} d\xi^{-\alpha}\right) + \mathrm{lot}
 $$
 and
 $$
 \Omega^{\alpha}_{g+th} =  \langle \alpha,\xi_+\rangle_{\kk^*}^{-1} d\xi^{\alpha} + \mathrm{lot},
 $$
 where $\mathrm{lot}$ represents terms of lower order in $t$. Therefore, with $2\rho = \sum_{\alpha\in \Phi_+} \alpha$ 
 we get
 \begin{eqnarray}
 \lim_{t\to +\infty} e^{-2\langle \rho,d_{\xi_+}(g+th)\rangle} \hspace{-3em}&& \bigwedge_{\alpha\in \Phi_+}  (\Omega^{-\alpha}_{g+th} \wedge \Omega^{\alpha}_{g+th}) = \nonumber \\
 & = & \hspace{-0.5em}  P(\xi_+)^{-1} \bigwedge_{\alpha\in \Phi_+} \left(\left(\omega^{-\alpha} -\langle \alpha,\xi_+\rangle_{\kk^*}^{-1} d\xi^{-\alpha}\right)\wedge d\xi^{\alpha}\right). \label{iiii}
 \end{eqnarray}
 Note that the right hand side of (\ref{iiii}) is never-vanishing.
\end{proof}

With this information in hand, we can now define our bundle of half-forms $\kappa_\bullet$: it is the trivial line bundle on $(T^\ast K)_{\mathrm{reg}}$, together with the map which sends the square of the trivializing section to $\widehat{\Omega}_\bullet$,
\[
  \kappa_{\bullet}^{\otimes 2} \cong \KK_\bullet, \qquad 1_{\kappa_\bullet}^{\otimes 2} \mapsto \widehat{\Omega}_\bullet .
\]
In view of this identity, we will also denote the trivializing section of the half-form bundle by
\[
 1_{\kappa_\bullet} = \widehat{\Omega}^{1/2}_\bullet .
\]

\begin{remark} The trivializing section considered in Proposition \ref{prop_trivialization} is only polarized for $\Pm_\mathrm{Sch}$; for the Kähler polarizations, it is the holomorphic section $\Omega_g$ itself which is polarized, whereas for $\Pm_\mathrm{KW}$ there is no global section which is simultaneously trivializing and polarized. For later use we collect explicit expressions for the polarized sections:
\end{remark}

\begin{proposition}\label{prop_polhf}
  A smooth section $\alpha \widehat{\Omega}^{1/2}_\bullet$ is polarized if
  \begin{itemize}
  \item[(a)] $\Pm_\bullet=\Pm_\mathrm{Sch}$ is Schrödinger and $\alpha = \pi^\ast f$ for $f$ a real function on $K$ and $\pi:T^\ast K \to K$ is the canonical polarization;
  \item[(b)] $\Pm_\bullet=\Pm_g$ is Kähler and $\alpha = e^{\langle \rho, \LL_g\circ \mu_\mathrm{inv} \rangle} (1+\det \Hess_{\tt} g)^{1/2} f$ for $f$ a $J_g$-holomorphic function;
  \item[(c)] $\Pm_\bullet=\Pm_\mathrm{KW}$ is Kirwin--Wu and $\alpha = \frac{h \circ \mu_\mathrm{inv}}{F_{\rho,A}}$ for $h\in C^\infty ( \breve{\tt}^\ast_+,\CC)$.
  \end{itemize}
\end{proposition}
\begin{proof}
For the Schrödinger and Kähler polarizations the condition we stated is actually necessary and sufficient, and is clear from Proposition \ref{prop_trivialization}; for the Kirwin--Wu polarization, it follows directly from Definition \ref{dfn_polsec} and the fact that the functions $F_{\rho,A}$ describe (as in Corollary \ref{cor_polcan}) certain holomorphic sections of the square root of
\[
 \KK_\phi^{-1} \cong L_{2\rho}\boxtimes L_{-2\rho} .
\]
\end{proof}

\section{The extended quantum bundle}
\label{sect-extendedqb}

 In this section we will show how the quantum states associated to the invariant K\"ahler polarizations described in Sections~\ref{sect-prelim}, \ref{sec_fibpol} and \ref{sect_KW} converge at infinite geodesic time to ${\mathcal P}_{\rm KW}$-polarized sections. This is interpreted in terms of generalized coherent state transforms in Section \ref{sect-cst}. One obtains an extended bundle of quantum states, as described in Section \ref{sect-introd}, with a flat connection which is, in fact, unitary along the family of polarizations given by the imaginary time flow  generated by the quadratic Casimir. This will be applied in Section \ref{sect-fourier} to relate the CST to the non-abelian Fourier transform so as to give a geometric interpretation of the Peter--Weyl theorem.

\subsection{Convergence of quantizations at infinite geodesic time}
\label{section5}

As outlined in Section \ref{sss_defquan}, for the discussion of convergence of quantum states associated to the limit of polarizations $\lim_{t \to \infty} \Pm_{g_t}=\Pm_{\mathrm{KW}}$ of Theorem \ref{thm_convpol}, we trivialize the prequantum line bundle $L \cong T^\ast K \times \CC$ as a hermitian line bundle: it follows from Proposition \ref{prop_Pg} that in the unitary frame, the $\Pm_g$-polarized sections of $L$ are of the form $f(x,\xi) e^{-\kappa_g(x,\xi)}$. Using Proposition \ref{prop_polhf}, we find:

\begin{proposition}\label{prop:polarized half-form} The polarized half-form corrected sections associated to the Kähler polarization $\Pm_g$ are of the form
  \[
   f e^{-\kappa_g} \otimes e^{\langle \rho, \LL_g \circ \mu_{\mathrm{inv}} \rangle}(1+\det \Hess_{\tt^\ast} g)^{1/2} \widehat{\Omega}^{1/2}_g \in C^\infty_{\Pm_g}(T^\ast K,L)\otimes C^\infty_{\Pm_g}(T^\ast K,\kappa_{g}),
  \]
  where $f$ is $J_g$-holomorphic and the prequantum line bundle $L$ is trivialized unitarily.
\end{proposition}

We now choose $f$ in the expression from Proposition~\ref{prop:polarized half-form} equal to the section of the bundle of quantizations $f = e^{-g(\lambda +\rho)} f^{g}_{\lambda,A}$, that is, we consider the sections
\begin{equation}\label{dfn_sglA}
 s^{g}_{\lambda,A} := e^{-g(\lambda +\rho)} f^{g}_{\lambda,A} e^{-\kappa_g} \otimes e^{\langle \rho, \LL_g \circ \mu_{\mathrm{inv}} \rangle}(1+\det \Hess_{\tt^\ast} g)^{1/2} \widehat{\Omega}^{1/2}_g 
\end{equation}
of the bundle of Kähler quantizations over the space of Kähler structures. We shall come back to the appearance of the constants $e^{-g(\lambda +\rho)}$ below when discussing the generalized coherent state transform.

The quantum states for the Kirwin--Wu polarization on the other hand have a distributional nature with respect to the value of the invariant moment map $\mu_{\mathrm{inv}}$. First, we note

\begin{proposition}
  The Bohr--Som\-mer\-feld fibers of the Kir\-win--Wu polarization are the preimages $\mu_{\mathrm{inv}}^{-1}(\lambda)$ for all dominant integral weights $\lambda \in \tt^\ast_\ZZ \cap \breve{\tt}^\ast_+$.
\end{proposition}

\begin{proof}
  Since our prequantum line bundle is trivializable with connection $\nabla -i \theta$, where $\theta$ is the Liou\-ville--1-form, which is $K\times K$-invariant, it suffices to check the monodromy condition which defines the Bohr--Som\-mer\-feld fibers on the symplectic cross-section $\Sigma \cong T\times \breve{\tt}^\ast_+$. Here, the Liouville form restricts to the standard toric form, and the resulting prequantum monodromy is $e^{2\pi i \langle \mu_\mathrm{inv}(p),\eta \rangle}$. Including the half-form bundle results in the total monodromy
  \[
   m(\eta,p) = e^{2\pi i \langle \mu_\mathrm{inv}(p)-\rho,\eta \rangle} ,
   \]
which proves the assertion.
\end{proof}

\begin{definition}
Denote by $\delta_{\mu_{\mathrm{inv}}^{-1}(\xi_+)}$ the $K\times K$-invariant Radon measure on $T^\ast K$ supported on $\mu_{\mathrm{inv}}^{-1}(\xi_+)$ with total mass 1; in other words,
\[
 \int_{T^\ast K} f\, d\delta_{\mu_{\mathrm{inv}}^{-1}(\xi_+)} =
 \int_{K \times K} f\left( (k_1,k_2)(e,\xi_+)\right) d k_1 d k_2 ,
\]
where $d k_i$ denotes the normalized Haar measure on $K$.

Then denote
\begin{equation}\label{dfn_sKWlA}
s^{\mathrm{KW}}_{\lambda,A} := (2\pi)^{r/2}P(\lambda+\rho)^2 F_{\lambda,A} \delta_{\mu_{\mathrm{inv}}^{-1}(\lambda+\rho)} \otimes \widehat{\Omega}^{1/2}_{\mathrm{KW}} .
\end{equation}
\end{definition}

\begin{lemma}[Laplace approximation] \item[(i)] Let $V$ be a real vector space; fix an inner product $\langle\ ,\ \rangle$ and denote by $d x$ the corresponding normalized Lebesgue measure.

Let $f$ be a smooth function on an open subset of $V$ with a unique non-degenerate minimum at $x_0$, and also assume that $f(x_0)=0$. Define the Hessian $\Hess_f(x)$ of $f$ at any point $x\in V$ to be the endomorphism occuring in the second order term of the Taylor expansion
\[
  f(x+\varepsilon \dot x) = f(x) + \varepsilon d_x f \dot x + \varepsilon^2 \langle \dot x , \frac{\Hess_f(x)}{2}\dot x \rangle + o(\varepsilon^2) .
\]
Then the measures
\[
 \sqrt{\det \frac{1}{2\pi}\Hess_f(x_0)} e^{-t f(x)}d x
 \overset{t \to \infty}{\longrightarrow} \delta_{x_0}
\]
converge to a Dirac delta supported at $x_0$ as $t$ goes to infinity.
\item[(ii)] Consider a convex $W$-invariant function $h:\tt^\ast \to \RR$; then the function
\begin{equation}\label{eqn_leading}
 \psi_h(\xi) = h(\lambda+\rho) - h(\xi) + \langle \xi-(\lambda+\rho),\LL_h(\xi) \rangle
\end{equation}
is convex on $\breve{\tt}^\ast_+$ and has a unique minimum at $\xi = \lambda+\rho$, with $h(\lambda+\rho)=0$ and
\[
\Hess_h(\lambda+\rho) = \Hess_{\psi_h}(\lambda+\rho) .
\]
\item[(iii)] Therefore, in particular as a measure on $(T^\ast K)_{\mathrm{reg}}$,
\[
\sqrt{\det \frac{t}{2\pi}\Hess_h(\lambda+\rho)}e^{-t {\psi_h} \circ \mu_{\mathrm{inv}}} \frac{\omega_\mathrm{std}^n}{n!}
\overset{t \to \infty}{\longrightarrow} P(\lambda+\rho)^2 \delta_{\mu_{\mathrm{inv}}^{-1}(\lambda+\rho)} ,
\]
where the Hessian is calculated with respect to an inner product in which $\tt^\ast_\ZZ$ has covolume one.
\end{lemma}

\begin{proof}
While (i) is well known, (ii) is a direct calculation which we omit. For (iii), we first note that it suffices to check against $K\times K$-invariant functions, since for any compactly supported continuous function $\phi$
\begin{align*}
 & \int_{(T^\ast K)_{\mathrm{reg}}} \phi \sqrt{\det \frac{t}{2\pi}\Hess_h(\lambda+\rho)}e^{-t \psi_h \circ \mu_{\mathrm{inv}}} \frac{\omega_\mathrm{std}^n}{n!} = \\
  & = \int_{(T^\ast K)_{\mathrm{reg}}} \left(\int_{K\times K} (k_1,k_2) \cdot \phi d k_1 d k_2\right) \sqrt{\det \frac{t}{2\pi}\Hess_h(\lambda+\rho)}e^{-t \psi_h \circ \mu_{\mathrm{inv}}} \frac{\omega_\mathrm{std}^n}{n!} . &
\end{align*}
For invariant functions, which we can write as $\phi\circ \mu_{\mathrm{inv}}$, integration against the symplectic volume form (\ref{vol_form}) can be simplified, using the counterpart of Weyl's integration formula for $\kk^\ast$ (cf. \cite{macdonald:1980}, \cite[Chap.~IX, \S~6.3 (12)]{Bourbaki})
\begin{align*}
& \int_{(T^\ast K)_{\mathrm{reg}}} (\phi\circ\mu_{\mathrm{inv}}) \sqrt{\det \frac{t}{2\pi}\Hess_h(\lambda+\rho)}e^{-t \psi_h \circ \mu_{\mathrm{inv}}} \frac{\omega_\mathrm{std}^n}{n!} \overset{\eqref{vol_form}}{=} \\
& = \int_{K} \int_{\kk_{\mathrm{reg}}^\ast} (\phi\circ\mu_{\mathrm{inv}}) \sqrt{\det \frac{t}{2\pi}\Hess_h(\lambda+\rho)}e^{-t \psi_h \circ \mu_{\mathrm{inv}}} dk_1  d\xi = \\
& = \int_{K\times K} \int_{\breve{\tt}^\ast_+} P(\xi_+)^2 \phi(\xi_+) \sqrt{\det \frac{t}{2\pi}\Hess_h(\lambda+\rho)}e^{-t \psi_h (\xi_+)} dk_1 dk_2 d\xi_+
\end{align*}
and the claim follows from (i) and (ii).
\end{proof}

\begin{theorem}\label{thm_convstates} For the geodesic ray $g_t = g+th$, on the regular subset $(T^\ast K)_{\mathrm{reg}}$
the family of quantum states defined in (\ref{dfn_sglA}) extends continuously to the Kirwin--Wu polarization,
\[
 \lim_{t \to \infty} s^{g_t}_{\lambda,A} = s^{\mathrm{KW}}_{\lambda,A} .
\]
\end{theorem}

\begin{proof}
  Using the fact that the factor containing the half-form
  \[
   \lim_{t \to \infty} \widehat{\Omega}^{1/2}_{g_t} =: \widehat{\Omega}^{1/2}_{\mathrm{KW}}
\]
converges by construction, we collect all other factors and rewrite them as
\[
 e^{-g_t(\lambda +\rho)}f_{\lambda,A}^{g_t} e^{-\kappa_{g_t}} e^{\langle \rho, \LL_{g_t} \circ \mu_{\mathrm{inv}} \rangle}(1+\det \Hess_{\tt^\ast} g_t)^{1/2} = F_1 F_2 F_3
\]
where 
\begin{align*}
 F_1(t) = f_{\lambda,A}^{g_t} e^{-\langle \lambda, \LL_{g_t} \circ \mu_{\mathrm{inv}} \rangle} \, & \to \, F_{\lambda,A} \\
 F_2(t) = \sqrt{\det \frac{t}{2\pi} \Hess_h(\lambda+\rho)}e^{- t \psi_h\circ \mu_{\mathrm{inv}}} \, & \to \, P(\lambda+\rho)^2 \delta_{\mu_{\mathrm{inv}}^{-1}(\lambda+\rho)} \\
 F_3(t) = e^{-\psi_g \circ \mu_{\mathrm{inv}} } (2\pi)^{r/2}\sqrt{\frac{\det \Hess_{\tt^\ast} h }{\det \Hess_h(\lambda+\rho)}+o(t^{-1})} \, & \to \, \widetilde{F}_3 .
\end{align*}
where $\widetilde{F}_3$ is a continuous function on $(T^\ast K)_{\mathrm{reg}}$ which is constant with value $(2\pi)^{r/2}$ on $\mu_{\mathrm{inv}}^{-1}(\lambda+\rho)$.
\end{proof}

\begin{remark}
Regarding degeneration to the Schrödinger polarization, it is immediate to verify that
\[
 \lim_{t \to 0} s^{th}_{\lambda,A} = \tr (\pi_{\lambda}(x)A) \otimes \widehat{\Omega}_{\mathrm{Sch}} .
\]
\end{remark}

\begin{corollary}
The sections $s^{\mathrm{KW}}_{\lambda,A}$ are Kirwin--Wu polarized.
\end{corollary}
\begin{proof}
Over the open subset where $F_{\rho,B} \neq 0$, trivially
\[
 s^\mathrm{KW}_{\lambda,A} = (2\pi)^{r/2}P(\lambda+\rho)^2 F_{\lambda,A}F_{\rho,B} \delta_{\mu_{\mathrm{inv}}^{-1}(\lambda+\rho)} \otimes \frac{1}{F_{\rho,B}}\widehat{\Omega}^{1/2}_{\mathrm{KW}} .
\]
Since it arises as a limit of polarized sections, and the Hamiltonian vector fields which generate the limit polarization are limits of polarized generators of the Kähler polarizations, the covariant derivatives are still 0 in the limit.
\end{proof}

\subsection{Generalized coherent state transforms and geometric quantization}
\label{sect-cst}
Following \cite{kirwin.mourao.nunes:2013} and Section \ref{section5}, we define the Hilbert space ${\mathcal H}_{g+th}$ for the quantization of $T^*K$ in the K\"ahler polarization ${\mathcal P}_{g+th}$, $t\geq 0,$ as the norm completion
$$
{\mathcal H}_{g+th} := \overline{\left\{ \sigma_{\lambda,A}^{g+th}\mid \lambda\in \hat K, A\in {\rm End}\, (V_\lambda) \right\}},
$$
where $\hat K$ denotes the set of equivalence clases of irreducible representations of $K$ labelled by highest weight $\lambda\in \tt_\ZZ^*$ and 
$$
\sigma_{\lambda,A}^{g+th} := e^{-(g+th)(\lambda+\rho)}f_{\lambda,A}^{g+th} e^{-\frac12 \kappa_{g+th}}\otimes \Omega_{g+th}^\frac12.
$$ This family of Hilbert spaces can be very usefully (and significantly) described by a \emph{generalized coherent state transform} (gCST) 
$$
C_{t,h}: \HH_{g} \to \HH_{g+th}
$$
as follows. (See \cite[(4.19)]{kirwin.mourao.nunes:2013} and \cite{kirwin.mourao.nunes:2014}.) One considers two operators, 
 naturally associated to the $\mathrm{Ad}^*$-invariant uniformly convex function $h$,  
 \begin{itemize}
     \item[\it (i)] 
 The \emph{prequantization}
 $\hat h$
 of $h$,
 \begin{equation}
 \label{ee-pr-hdouble}
 \hat h := \left(i \nabla_{X_h} + h \right) \otimes 1 
 + 1 \otimes i L_{X_h}  \, 
 \end{equation}
 which does not preserve $\HH_{g}$. ($L_{X_h}$ denotes the Lie derivative with respect to $X_h.$)
 
\item[\it (ii)] The \emph{quantization} of $h$, $\Qm(h)$, which
 preserves the spaces $\HH_g$ and
 we define as in \cite[(1.6)]{kirwin.mourao.nunes:2014}
 \begin{equation}
 \label{quantumop}
 \Qm(h) \, \sigma_{\lambda, A}^g := h(\lambda+\rho) \, \, \sigma_{\lambda, A}^g , 
 \end{equation}
 so that the sections $\sigma_{\lambda, A}^g$ are eigensections 
 of $Q(g)$ with eigenvalues $h(\lambda+\rho)$.
 \end{itemize}

The gCST is defined by 
 \begin{equation}
 \label{dd18double}
 C_{t,h} := e^{ t\hat h} \circ e^{- t\Qm(h)} \quad : \, \, \HH_g 
 \longrightarrow \HH_{g+th}  \, .
 \end{equation}
so that (see Section 4 in \cite{kirwin.mourao.nunes:2013}) 
$$
\sigma_{\lambda,A}^{g+th} = C_{t,h} \sigma_{\lambda,A}^g, \,\, \forall \lambda \in \hat K, , A\in {\rm End}\,(V_\lambda).
$$

In fact, the gCST can be also applied directly to the Hilbert space of the Schr\"odinger polarization
$$
\HH_{\rm Sch} :=  \overline{\left\{ \sigma_{\lambda,A}^{0} \mid \lambda\in \hat K, A\in {\rm End}\, (V_\lambda) \right\}},
$$
where (see \cite{hall:2002})
$$
\sigma_{\lambda,A}^{0} := \tr (\pi_\lambda(x) A) \otimes \sqrt{dx}.
$$
One has, indeed,
$$
C_{1,g} \left(\tr (\pi_\lambda(x) A) \otimes \sqrt{dx}\right) = \sigma^g_{\lambda,A},\,\, \forall \lambda\in \hat K, A\in {\rm End}\,(V_\lambda).
$$
Note that $C_{t,h}, C_{1,g}$ are linear isomorphisms of Hilbert spaces that intertwine the natural actions of $K\times K$ on $\HH_{\rm Sch}, \HH_g, \HH_{g+th}.$
One then has the following restatement of the convergence of holomorphic sections to distributional sections stated in Theorem \ref{thm_convstates},
\begin{theorem}\label{alacst}In the distributional sense, 
$$
\lim_{t\to +\infty} C_{t,h} \sigma_{\lambda,A}^{g+th} = s^{\rm KW}_{\lambda,A}, \, \forall \lambda\in \hat K, A\in {\rm End}\,(V_\lambda).
$$
This result also holds if we apply the gCST directly to the Schr\"odinger quantization, as above. (Just set the initial symplectic potential $g=0$.)
\end{theorem}

\begin{remark}
Note that the particular form of the quantum operator  $\Qm (h)$ is crucial in obtaining Theorem \ref{alacst}. An analogous result holds in the quantization of symplectic toric manifolds, 
where a similarly defined gCST relates quantizations in a Mabuchi geodesic family holomorphic toric polarizations with the quantization in the real toric polarization which is attained at infinite geodesic time.
(See \cite{baier.florentino.mourao.nunes:2011, kirwin.mourao.nunes:2013b, kirwin.mourao.nunes:2016}.) 

The definition of $\Qm(h)$ is related to the Duflo isomorphism. Recall that the Poincaré--Birkoff--Witt theorem gives a linear isomorphism between the symmetric algebra ${\mathcal S}(\kk)$ of $\kk$ and its universal enveloping algebra ${\mathcal U}(\kk)$, $PBW: {\mathcal S}(\kk)\to  {\mathcal U}(\kk)$. This map is not a morphism of algebras but its restriction to the subspace of invariants, $PBW: {\mathcal S}(\kk)^\kk\to  {\mathcal U}(\kk)^\kk$, composed with a map $D:{\mathcal S}(\kk)^\kk\to {\mathcal S}(\kk)^\kk$ known as the Duflo isomorphism, gives an isomorphism of algebras $PBW\circ D: {\mathcal S}(\kk)^\kk\to {\mathcal U}(\kk)^\kk.$ As already mentioned in \cite{kirwin.mourao.nunes:2014}, when $h$ is the quadratic Casimir this corresponds to adding a constant to the bi-invariant Laplacian on $K$ which corresponds to the shift 
$C_2(\lambda)\to C_2(\lambda)+\frac12\langle\rho,\rho\rangle_{\kk^*}=\frac12(\lambda+\rho)^2.$ This case is further discussed below in Section \ref{sect-connection}.
\end{remark}

In the next section, we discuss a flat connection on the extended bundle of quantum states and, in particular, we will consider the case when $h$ (and also $g$) is the quadratic Casimir so that $C_{t,h}$ becomes the coherent state transform of Hall and parallel transport on the extended bundle of quantum states is unitary.

\subsection{The flat connection on the extended bundle of quantum states}
\label{sect-connection} In the case where $h$ is given by the quadratic Casimir on $\kk$, 
$$h(\xi)= \frac12 \vert\vert\xi\vert\vert^2,$$
the map $C_{t,h}$  can be identified (with appropriate conventions, see \cite{hall:2002,florentino.matias.mourao.nunes:2005,florentino.matias.mourao.nunes:2006}) with the coherent state transform of Hall
\[
C_t:L^2(K,dx)\to {\mathcal H}L^2(K_\CC, d\nu_t), \qquad
C_t = {\mathcal C}\circ e^{\frac{t}{2}\Delta},
\]
which takes values in the space of holomorphic functions on $K_\CC$ with appropriate growth at infinity; here $d\nu_t$ is the avearged heat kernel measure, and $\Delta$ is the Laplacian for the bi-invariant metric on $K$ and ${\mathcal C}$ denotes analytic continuation.
This is, remarkably, a unitary isomorphism of Hilbert spaces \cite{hall:1994,HZ09}.
The matrix elements of the irreducible representation of $K$ with highest weight $\lambda$ are eigenvectors of $\Delta$ with eigenvalue 
$$
\frac12 \langle\lambda+\rho, \lambda+\rho\rangle_{\kk^*} - \frac12 \langle\rho,\rho\rangle_{\kk^*} = \Qm(h) +\frac12 \langle\rho,\rho\rangle_{\kk^*}.
$$
For more general $\mathrm{Ad}^*$-invariant uniformly convex functions $h$, 
$C_{t,h}$ is not unitary although is asymptotically unitary as $t\to +\infty$ (see Section~IV.A of \cite{kirwin.mourao.nunes:2014}).

Consider the following vector space of ${\mathcal P}_{\rm KW}$-polarized sections
$$
V_{\rm KW} := \left\{ \sigma_{\lambda,A}^{+\infty} \mid \lambda\in \hat K, A\in {\rm End}\, (V_\lambda) \right\},
$$
where $\sigma^{+\infty}_{\lambda,A}:= s^{\rm KW}_{\lambda,A}$. We will motivate the definition of an  inner product structure on $V_{\rm KW}$ as follows.
Let
$$
h(\xi):= \frac12 \vert\vert\xi\vert\vert^2.
$$
From \cite{hall:2002, florentino.matias.mourao.nunes:2005,florentino.matias.mourao.nunes:2006},
$$C_{t,h}: {\mathcal H}_{\rm Sch} \to {\mathcal H}_{th}$$ is a unitary isomorphism for all $t>0$, such that 
$$
C_{t,h}\sigma^0_{\lambda, A} = \sigma^{th}_{\lambda, A}.
$$

Let $\tr (\pi_\lambda(x) E_{ij}) =: \pi_\lambda(x)_{ij}, i,j=1, \dots, d_{\pi_\lambda}$ where $d_{\pi_\lambda}=\dim V_\lambda,$ denote the matrix elements of the representation $\lambda\in \hat K$ with respect to an orthonormal basis 
for a $K$-invariant inner product on $V_\lambda$, such that $\pi_\lambda (x)$ is unitary for $x\in K$. Then, moreover,
$$
\left\{ \sigma^0_{\lambda, E_{ij}} \right\}_{\lambda\in \hat K, i,j=1,\dots, d_{\pi_\lambda}}
$$
is an orthogonal basis for ${\mathcal H}_{\rm Sch}$ with (see, for example, \cite{kirwin.mourao.nunes:2013})
$$
\vert\vert \sigma^0_{\lambda, E_{ij}}\vert\vert^2 = (d_{\pi_\lambda})^{-1}.
$$

Motivated by the unitarity of $C_{t,h}, t>0,$ we will then define the inner product structure on $V_{\rm KW}$ by declaring that
$$
\left\{ \sqrt{d_{\pi_\lambda}} \sigma^{+\infty}_{\lambda, E_{ij}} \right\}_{\lambda\in \hat K, i,j=1,\dots, d_{\pi_\lambda}}
$$
is an orthonormal basis. Taking the norm completion, we define the Hilbert space 
$$
{\mathcal H}_{\rm KW} := \overline{V_{\rm KW}}.
$$

\begin{remark}
\item[(i)] Given the asymptotic unitarity of $C_{t,h}$ as $t\to +\infty$ (Section IV.A in \cite{kirwin.mourao.nunes:2014})  the inner product structure on ${\mathcal H}_{\rm KW}$ can equivalently be defined by taking asymptotic values along any Mabuchi ray $g+th, t>0$.
\item[(ii)] Note that the Hilbert spaces ${\mathcal H}_{th}, t>0,$ are reproducing kernel Hilbert spaces with reproducing kernel
$$
K_{t,h}((x,\xi),(x',\xi'))=\sum_{\lambda\in \hat K, i,j=1,\dots, d_{\pi_\lambda}} d_{\pi_\lambda} 
\sigma^{th}_{\lambda, E_{ij}} (x',\xi') \bar \sigma^{th}_{\lambda, E_{ij}} (x,\xi) ,
$$
see e.g. \cite[Prop.~2.2]{HZ09}.
The inner product on ${\mathcal H}_{\rm KW}$ then has a reproducing kernel obtained by taking the limit 
$$
K_{\rm KW}((x,\xi),(x',\xi')):=\lim_{t\to +\infty}  K_{t,h}((x,\xi),(x',\xi')) ,
$$
in the appropriate sense.
\end{remark}

As described in the introduction, we therefore obtain an extended bundle of Hilbert spaces of quantum states,
$$
\overline{\mathcal H} \to \overline{\mathcal M},
$$
where  $\overline{{\mathcal M}}:={\mathcal M}\cup \{\Pm_\mathrm{Sch},\Pm_\mathrm{KW}\}$, with ${\mathcal M}$ the space of Weyl-invariant convex functions on $\tt^*.$ The fiber of $\overline{\mathcal H}$ over $\Pm_\mathrm{Sch}$ (respectively $\Pm_\mathrm{KW}$) is ${\mathcal H}_{\rm Sch}$ (respectively ${\mathcal H}_{\rm KW}$), while over $g\in {\mathcal M}$ the fiber is ${\mathcal H}_{g}$.

In the spirit of \cite{axelrod.dellapietra.witten:1991}, the gCST then defines the parallel transport for a flat  $K\times K$-invariant connection, $\nabla^Q$, on $\overline{\mathcal H}$ for which the fiber generating sections $\sigma_{\lambda,A}, \lambda\in \hat K, A\in {\rm End}\, (V_\lambda)$, are parallel, where
$$
\sigma_{\lambda,A}({\mathcal P}) := \left\{ 
\begin{array}{rl}
\sigma_{\lambda,A}^0,& \,\, {\mathcal P}={\mathcal P}_{\rm Sch},\\
\sigma_{\lambda,A}^g,& \,\,{\mathcal P}={\mathcal P}_g\in {\mathcal M},\\
\sigma_{\lambda,A}^{+\infty},& \,\,{\mathcal P}={\mathcal P}_{\rm KW}
\end{array}\right. .
$$

Explicitly, for a section of ${\overline{\mathcal H}}$,
$$
s = \sum_{\lambda\in \hat K, i,j=1\dots, d_{\pi_\lambda}} 
a_{\lambda,ij} \sigma_{\lambda,E_{ij}} ,
$$
for a tangent vector $\delta h$ to ${\mathcal M}$ we obtain
$$
\nabla^Q_{\delta h} s =  \sum_{\lambda\in \hat K, i,j=1\dots, d_{\pi_\lambda}} 
\delta h(a_{\lambda,ij}) \sigma_{\lambda,E_{ij}}. 
$$
Note that, generalizing Section 2.4 in \cite{florentino.matias.mourao.nunes:2005}, a section on 
$\mathcal H =\overline{\mathcal H}\vert_{\mathcal M}$ of the form
$$
\sigma(g;x,\xi):= f(g;xe^{d_\xi g}) e^{-\frac12\kappa_g}\otimes \sqrt{\Omega_g},  
$$
where, for fixed $g$, $f$ is holomorphic in $\KK_{\mathbb C}$, is $\nabla^Q-$parallel iff for all $\delta h$ tangent to $\mathcal M$ at $g$,
$$
\delta_h f = -\Qm(h) f.
$$
Here, $\Qm(h)$ acts as in (\ref{quantumop}) given the isotypical decomposition of $f$.

Let now $\Upsilon\subset \overline{{\mathcal M}}$ be the closed Mabuchi geodesic ray generated by $h$ equal to the quadratic Casimir. From above, $\nabla^Q$ is unitary on the restriction $\overline{{\mathcal H}}\vert_\Upsilon$, such that its unitary parallel transport is given by the  CST of Hall. 
This establishes the ``equivalence" of this one-parameter family of quantizations of $T^*K$ where, notably, at one end of the interval we have the real Schr\"odinger polarization and at the other end we have the mixed Kirwin--Wu polarization. Note again that, in general, for other choices of invariant symplectic potential $h$, the gCST will not be unitary and will not give such an equivalence between quantizations.

\begin{remark}
Note that Huebschmann in \cite[Thm.~5.3 \& Thm.~6.5]{Huebschmann08} describes the CST between the vertical polorization and the Kähler polarizations using a holomorphic version of the Peter--Weyl theorem and the BKS transformation. The relation between the CST and the BKS pairing was also established in \cite{hall:2002,florentino.matias.mourao.nunes:2006}.
These works do not consider the limit case of the Kirwin--Wu polarization and the associated operator-valued Fourier transform.
\end{remark}

\section{The unitary non-abelian Fourier transform: the geometric link between the Peter--Weyl and Borel--Weil theorems.}
\label{sect-fourier}
 
The operator-valued Fourier transform (\cite[\S~18.8.1]{Dixmier77}, \cite[\S~5.3]{folland:1995}) for the compact group $K$ maps $L^2(K,dx)$ isometrically onto the Hilbert space direct sum $$\widehat{\bigoplus}_{\lambda\in \hat K} \mathrm{End}(V_\lambda),$$ where $\mathrm{End}(V_\lambda)$ is equipped with the renormalized Hilbert--Schmidt inner product $(A,B)\mapsto d_{\pi_\lambda}\tr(A^*B)$.  More precisely, for $F\in L^1(K)$ one sets 
$$\hat F(\pi_\lambda):=\int_K F(x)\pi_\lambda(x)^*dx,$$
and obtains the Plancherel formula
$$\|F\|_{L^2(K,dx)}^2= \sum_{\lambda\in \hat K} d_{\pi_\lambda} \tr(\hat F(\pi_\lambda)^*\hat F(\pi_\lambda))$$
for $F\in L^2(K,dx)\subseteq L^1(K,dx)$.

We want to relate the parallel transport from ${\mathcal H}_{\rm Sch}$ to ${\mathcal H}_{\rm KW}$ in the extended quantum bundle $\overline{\mathcal H}$, as described in Section 6,
$$
L^2(K,dx)\cong {\mathcal H}_{\rm Sch} \ni f= \sum_{\lambda\in \hat K} \tr(\pi_\lambda(x)A)\otimes \sqrt{dx} \mapsto \sum_{\lambda\in \hat K} \sigma^{+\infty}_{\lambda,A} \in {\mathcal H}_{\rm KW}, 
$$
to the operator-valued Fourier transform. For this we need to define a $K\times K$-equivariant unitary isomorphism
$$
\Phi: {\mathcal H}_{\rm KW} \cong  \widehat{\bigoplus}_{\lambda\in \hat K} \mathrm{End}(V_\lambda).
$$
It is clear that this is fixed up to a choice of unitary phase on each irreducible subspace $V_\lambda\otimes V_\lambda^\ast$, so that we will just define this isomorphism to be
$$
\Phi (\sigma^{+\infty}_{\lambda,A}) = A.
$$
This realizes the unitary non-abelian operator valued Fourier transform as a transform with values in ${\mathcal H}_{\rm KW}$ and which, by the obvious identification ${\mathcal H}_{\rm Sch}\cong L^2(K,dx)$, coincides with the parallel transport on the extended quantum bundle restricted to $\Upsilon$ and given by the coherent state transform of Hall at infinite time, as described in Section \ref{sect-connection}.

To further justify this definition and to give the geometric link between the Peter--Weyl and Borel--Weil theorems let now $u\in V_\lambda, v^*\in V_\lambda^*$ and consider the simple tensor 
$A = u\otimes v^*\in \mathrm{End}(V_\lambda).$ We have from (\ref{dfn_sKWlA})
$$
\sigma^{+\infty}_{\lambda,u\otimes v^*} = (2\pi)^{r/2} P(\lambda+\rho)^2 F_{\lambda, u\otimes v^*} 
\delta_{\rm inv}^{-1}(\lambda+\rho) \hat\Omega_{\rm KW}^\frac12.
$$
From Remark \ref{products} recall that 
$$
F_{\lambda, u\otimes v^*} (x,\xi)= \tr (\pi_\lambda(x_1) v_\lambda\otimes v^*) \tr (\pi_\lambda(x_2^{-1}) u\otimes v_\lambda^*) .
$$
In line with the results in Section 4, it is clear that $F_{\lambda, u\otimes v^*}$ is a $T$-equivariant function on $K$  which identifies it with a section of the Borel--Weil bundle $L_\lambda \boxtimes L_{\lambda^*}$ over the product of coadjoint orbits $K/T \times T\backslash K$, where $\lambda^*$ denotes the highest weight of the conjugate representation $V_\lambda^\ast$.
Indeed, for $t\in T$,
$$
\tr (\pi_\lambda(x_1 t) v_\lambda\otimes v^*) = \chi_\lambda(t)\tr (\pi_\lambda(x_1) v_\lambda\otimes v^*)
$$
and
$$
\tr (\pi_\lambda((tx_2)^{-1}) u\otimes v_\lambda^*)=\tr (\pi_\lambda(t^{-1}x_2^{-1}) u\otimes v_\lambda^*) = \chi_{\lambda^*}(t)\tr (\pi_\lambda(x_2^{-1}) u\otimes v_\lambda^*).
$$
This identifies directly $F_{\lambda, u\otimes v^*}$ with $u\otimes v^*\in \mathrm{End} (V_\lambda),$ as is implicit in the map $\Phi$.
In other words, geometric quantization provides a natural link between $\tr(\pi_\lambda(x)u\otimes v^*)\in L^2(K,dx)$ and the distributional section $\sigma^{+\infty}_{\lambda,u\otimes v^*}$ supported on the Bohr--Sommerfeld cycle $\mu_{\rm inv}^{-1}(\lambda+\rho)$ and identified with the corresponding Borel--Weil section in $H^0({\mathcal O},L_\lambda)\otimes H^0({\mathcal O},L_{\lambda^*})\cong \mathrm{End}(V_\lambda).$

\section{An outlook on a more general theory}
\label{sec_conjectures}
In this section we outline a program that aims at assigning geometric cycles to representations in the context of K\"ahler Hamiltonian $G$-manifolds, in a construction that parallels, and generalizes, the paradigmatic case of $T^*K$, $G=K\times K$ that is detailed in this paper.  We will address the more general case in future work. Some of the ideas involved were already alluded to in Section \ref{sec_Fourierpol}.

Let $G$ be a compact, connected and simply-connected Lie group and let $M$ be a K\"ahler manifold with an Hamiltonian action of $G$ such that the K\"ahler structure is $G$-invariant. Let $G \circlearrowright (M,\omega) \overset{\mu}{\to} \gg^\ast$ be the moment map.

Associated to the Hamiltonian $G$-action on $M$, there is an invariant moment map $\mu_\mathrm{inv}:\breve{M} \to \aa^\ast_+$ which on an open dense subset $\breve{M} \subset M$ defines an action of a torus $T_{\mathrm{inv}}$ with $\Lie(T_{\mathrm{inv}}) = \aa$. Here, $\aa^\ast$ is a certain quotient of the dual $\tt^\ast$ of the Lie algebra of the maximal torus of $G$ (as is already the case for $M=T^\ast K$, where the group $K\times K$ acts, whereas the Kirwan polytope lies in $\tt^\ast$). 

Recall, from Section \ref{sec_Fourierpol}, the space of invariant smooth functions 
\begin{equation}\label{invfunctions}
\mu_{\textrm{inv}}^\ast \left( C^\infty(\tt^\ast)^W \right) =
 \mu^\ast \left( C^\infty(\gg^\ast)^G \right),    
\end{equation}
which has a dense subset  corresponding to the polynomial ring $\mathcal{Z}_G(\mathcal{U}(\gg))$ generated by $r$ independent Casimirs of $\gg$. Each $h\in C^\infty(\gg^*)^G$ determines a smooth function on $M$, $h\circ \mu = h\circ \mu_{\mathrm{inv}}$ which we will for simplicity  also denote by $h$. 
From Lemma \ref{lemma_casimirs}, we then have the Diagram (\ref{diag_invpol}) of Poisson varieties in the sense of Ortega (with no complex structures involved yet).

We now assume that standard prequantization geometric data are $G$-invariant and, in particular, we assume that the $G$-action lifts to the prequantum line bundle $L\to M.$ 
Let $\mathcal{P}$ be the $G$-invariant K\"ahler 
polarization which we assume to admit a half-form bundle $\kappa_\mathcal{P}$.  The Lie algebra $\gg$ acts on $\kappa_\mathcal{P}$ via the partial connection defined by taking Lie derivatives along  fundamental vector fields for the $G$-action. One will then have an isotypical decomposition for the Hilbert space of half-form corrected quantum states
$$
\mathcal{H}_{\mathcal{P}} := H^0(M,L\otimes \kappa_\mathcal{P}) =
\bigoplus_{\lambda \in \hat G}  \mathcal{H}_{\mathcal{P}}^\lambda.
$$

\begin{conjecture}\label{conj1}
Let $h\in \mu^\ast \left( C^\infty(\gg^\ast)^G \right)$ be uniformly convex. Then, the flow in imaginary time of the Hamiltonian vector field $X_h$ generates a Mabuchi ray of $G$-invariant K\"ahler polarizations on $M$, $\mathcal{P}_s, s\geq 0$, with initial point $\Pm_J$, and which at infinite geodesic time defines a $G$-invariant mixed polarization $\mathcal{P}_{\mu_\mathrm{inv}}=\mathcal{P}_{\mu_\mathrm{inv},J,h}$, with real directions given by the orbits of $T_{\mathrm{inv}}$. Note that in the non-mul\-ti\-pli\-city-free case, the limit polarization a priori might depend on both the initial K\"ahler polarization $\Pm_{J}$ and the direction specified by $h$.
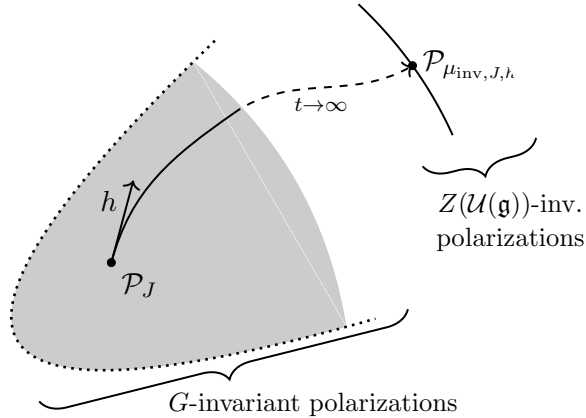
\[
\begin{tikzpicture}[scale=1.5]
  \begin{axis}[axis lines=none,axis equal,grid=both,no marks,domain=-2:2,xmax=2.5,xmin=-2.5,ymax=7,ymin=-1,samples=100,rotate=-60,]
  \addplot[dotted, line width=1pt,name path=a] {(x^2)};
  \addplot[dotted, line width=0pt,name path=b] {(4-0.1*x^2)};
  \addplot[fill=none] fill between[of=a and b,split,
    every segment no 1/.style={fill,gray,opacity=.4},];
  \addplot[line width=0.75pt,domain=-1.5:0.5] {(6-0.1*x^2)};
  \draw [fill, black] (axis cs: -0.5,5.975) circle [radius=1pt] node[right,name=Pmuinv] {$\Pm_{\mu_{\textrm{inv},J,h}}$};
  \draw [fill, black] (axis cs: -0.25,1.5) circle [radius=1pt] node[below right,name=PJ] {$\Pm_{J}$};
  \draw[line width=0.75pt,->] (axis cs: -0.25,1.5) to node[pos=0.75,left] {$h$} (axis cs: -1,2.25);
  \draw[line width=0.75pt] (axis cs: -0.25,1.5) to [out=135,in=275] (axis cs: -1.1,3.75);
  \draw[line width=0.75pt,->,dashed] (axis cs:-1.1,3.75) to [out=95,in=270] node[pos=.5,below] {${\tiny {}_{t \to \infty}}$} (axis cs:-0.5,5.975);
\draw[decorate, line width=0.75pt, decoration={brace, amplitude=2ex, raise=1ex}] (axis cs: 1.95,4.45) -- (axis cs: 0.65,-0.1) node[pos=.25, below=5ex] {{\small $G$-invariant polarizations}};  
\draw[decorate, line width=0.75pt, decoration={brace, amplitude=2ex, raise=1ex}] (axis cs: 0.9,6.9) -- (axis cs: 0.5,5.5) node[pos=.2, below=4ex, text width=3cm, text centered] {{\small $Z(\mathcal U(\gg))$-inv. polarizations}};
  \end{axis}
\pgfresetboundingbox
\path[use as bounding box] (3,-2) rectangle (10.5,1.75);
\end{tikzpicture}
\]
\end{conjecture}

We will assume that $\mathcal{P}_{\mu_\mathrm{inv}}$ admits a half-form bundle which is compatible with the degeneration. Definition \ref{defBS} gives that the 
half-form corrected Bohr--Sommerfeld cycles will be of the form 
$$
M_\lambda = \mu_{\mathrm{inv}}^{-1} (\lambda +a),
$$
for highest weights $\lambda\in \mu_{\mathrm{inv}}(M)\cap \tt^*_{\mathbb Z}$ and $a\in \tt^*_+$, where the shift $a$ is due to the contribution from the holonomy along the half-form bundle.
The quotients $M_\lambda/T_\mathrm{inv}, \lambda\in \mu_{\mathrm{inv}}(M)\cap \tt^*_{\mathbb Z}$, provide the set of quantizable symplectic reductions of $M$ (including the half-form correction) with respect to the Hamiltonian action of $T_\mathrm{inv}$.
We call the sets $M_\lambda$ \emph{spectral submanifolds} since for $\mathcal{P}_{\mu_\mathrm{inv}}$-polarized half-form corrected sections the prequantum operator for $h\in \mu^\ast \left( C^\infty(\gg^\ast)^G \right)$ corresponds just to the multiplication by $h(\lambda+a).$
 
We will have therefore a decomposition 
 \begin{equation}\label{invisot}
\mathcal{H}_{\mathcal{P}_{\mu_\mathrm{inv}}} = 
\bigoplus_{\lambda \in \hat G}  \mathcal{H}_{\mathcal{P}_{\mu_\mathrm{inv}}}^\lambda,\end{equation}
where elements in $\mathcal{H}_{\mathcal{P}_{\mu_\mathrm{inv}}}^\lambda$ are distributional sections of $L\otimes \kappa_{\mathcal{P}_{\mu_\mathrm{inv}}}$ supported on $M_\lambda.$ Note that since $\mathcal{P}_{\mu_\mathrm{inv}}$ is $G$-invariant, $\mathcal{H}_{\mathcal{P}_{\mu_\mathrm{inv}}}$ will have an isotypical decomposition into irreducible representations of $G$.

Over the infinite-dimensional family of $G$-invariant K\"ahler polarizations we will have a bundle of K\"ahler quantizations. We conjecture that this bundle extends to include the 
quantizations for the  polarizations arising at infinite geodesic times $\mathcal{P}_{\mu_\mathrm{inv}}$, as described above. 
As presented in detail in Sections \ref{sect-cst} and \ref{sect-connection} in the case of $M=T^*K, G=K\times K$,
in the spirit of \cite{axelrod.dellapietra.witten:1991}, 
the extended quantum bundle should be equipped with a flat connection whose parallel transport is defined, along the  Mabuchi ray generated by $h$ in Conjecture \ref{conj1}, by a generalized coherent state transform $C_{s,h}$. This gCST is given by the composition of the operator of analytic continuation along the Hamiltonian flow in imginary time defining the Mabuchi ray and a quantum operator, as follows.

Let $\hat h$ be the Kostant--Souriau prequantum operator associated to $h$. To $h$ we also associate a ``quantum" operator defined by
\[
\mathcal{Q}(h) : \mathcal{H}_\mathcal{P}\to \mathcal{H}_\mathcal{P}, \qquad
\mathcal{Q}(h) \sigma := h(\lambda+a) \sigma \text{ if }
\sigma \in \mathcal{H}_{\mathcal{P}}^\lambda.
\]
The gCST $C_{s,h}: \mathcal{H}_\mathcal{P}\to \mathcal{H}_{\mathcal{P}_s}$ is then defined as 
\[
C_{s,h} := e^{s\hat h} \circ e^{-s \mathcal{Q}(h)}.
\]
For each $s\geq 0$, the gCST is $G$-equivariant, so in particular preserves the isotypical components $\mathcal{H}_{\mathcal{P}_s}^\lambda$.
\begin{conjecture}\label{conj2}
 For each $\sigma\in \mathcal{H}_\mathcal{P}^\lambda$,
$$
\lim_{s\to +\infty} C_{s,h}\sigma =: \sigma_\infty \in \mathcal{H}_{\mathcal{P}_{\mu_\mathrm{inv}}}^\lambda.
$$
Moreover, the map $\mathcal{H}_\mathcal{P}^\lambda\ni\sigma \mapsto \sigma_\infty \in \mathcal{H}_{\mathcal{P}_{\mu_\mathrm{inv}}}^\lambda$ is $G$-equivariant.
\end{conjecture}

Conjecture \ref{conj2} means that the isotypical decomposition of $\mathcal{H}_\mathcal{P}$ corresponds
to the decomposition of $M$ into the set of spectral submanifolds $M_\lambda$ and, correspondingly, to the set of quantizable symplectic reductions of $M$ with respect to the action of $T_\mathrm{inv}.$
Each representation of $G$, with highest weight $\lambda$, which is initially realized in $H^0(M,L\otimes \kappa_\mathcal{P})$ in terms of holomorphic sections, will thus be related with distributional sections of $L\otimes \kappa_{\mathcal{P}_{\mu_\mathrm{inv}}}$ supported on the spectral submanifold $M_\lambda$. 

\appendix

\section{Families of polarizations and convergence of half-form corrected quantizations} \label{app_conv}

\subsection{Continuous families of polarizations}\label{par_families} Since one way of describing a polarization $\Pm$ is in terms of a Lagrangian subbundle $\Pm \subset T^\ast_{\CC} M$, the set of polarizations carries a natural topology -- a family of polarizations on some index set $I$ is a family of Lagrangian subbundles $\Pm_t$ indexed by $t\in I$, and it is continuous if it is so as a vector bundle on $M\times I$.

Equivalently, as we are interested in pointwise or locally uniform convergence we may assume that $TM=M\times V$ and $TM_\mathbb C=M\times V_\mathbb C$. Thus $\mathcal P$ is completely determined by
$$M\to \mathrm{Gr}_k(V_\mathbb C), \quad x\mapsto [\mathcal P_x],$$ 
where $\dim V=2k$. Thus, locally $\mathcal P$ is embedded into $C(M, \mathrm{Gr}_k(V_\mathbb C))$, which equips $\mathcal P$ with a natural topology. Suppose now that a family $f_t$ of $\mathcal P_t$-polarized functions (in general position) converges to a family $\{f\}$ (in general position) of $\mathcal P$-polarized functions, then the joint kernels of the $df_t$ converge to the joint kernels of the $df$. But this means that $\mathcal P_t$ converges to $\mathcal P$ in the described topology.

Since taking exterior products is continuous, it follows in particular that in this case also the relative canonical bundle $\KK_{\Phi}$ (defined on $M\times I$ with respect to the projection $\Phi:M\times I \to I$ exactly as in (\ref{def_relcan})) defines a continuous family of line bundles 
\[
\begin{tikzcd}
\KK_\Phi \ar[r, phantom, "\subset"] \ar[d] & \bigwedge^n T^\ast_\CC M \ar[d] \\
M\times I \ar[r, equals] & M\times I
\end{tikzcd}
\]
which moves inside the vector bundle (constant with respect to the base $I$ indexing the family) of complexified $n$-forms. 
Accordingly, a half-form bundle for the family $\Pm_\bullet$ is a complex line bundle $\kappa_{\Phi}$ on $M\times I$ together with an isomorphism $\kappa_{\Phi}^{\otimes 2} \cong \KK_{\Phi}$.

In the case at hand in our paper, the canonical bundles of all polarizations are trivializable, which substantially simplifies the discussion; we therefore do not pursue the more general case here.

A family of polarizations parameterized by $I$ defines a family of quantum states $\Qm$, where
\[
  \Qm_{t} \subset \HH^{-\infty}_{\Pm_t}(M) 
\]
is a family of subspaces defined by a certain finite energy condition.
If $\Pm$ is a smoothly varying family of Kähler polarizations on a compact manifold such that the dimension of the space of solutions (which is finite dimensional in this case) remains constant, then by elliptic regularity this family actually comes from a vector bundle $\Qm \to I$.

\subsection{Convergence of quantum states} The convergence of quantum states in a family of polarizations is to be understood in the terms just laid out: a family of quantum states $s_t \otimes \sqrt{\Omega_t}\in \HH_{\Pm_t}^\infty(M)$ converges to $s_\infty \otimes \sqrt{\Omega_\infty}\in \HH_{\Pm_\infty}^{-\infty}(M)$ if $M$ can be covered by open subsets $U\subset M$ on which they are polarized products
\begin{align*}
 (s_t \otimes \sqrt{\Omega_t}) \vert_U & = s_t^U \otimes \sqrt{\Omega_t^U} \in C^\infty_{\Pm_t}(U,L) \underset{C^\infty_{\Pm_t}(U,\CC)}{\otimes} C^\infty_{\Pm_t}(U,\kappa_{\Pm_t}) , \\
 (s_\infty \otimes \sqrt{\Omega_\infty}) \vert_U & = s_\infty^U \otimes \sqrt{\Omega_\infty^U} \in C^{-\infty}_{\Pm_\infty}(U,L) \underset{C^\infty_{\Pm_\infty}(U,\CC)}{\otimes} C^\infty_{\Pm_\infty}(U,\kappa_{\Pm_\infty}) ,
\end{align*}
and furthermore using the natural injection $C^\infty(L) \hookrightarrow C^{-\infty}(L)$ provided by the fixed Hermitian structure on $L$ we have
\[
\lim_{t \to \infty} s^U_t = s^U_\infty \in C^{-\infty}(U,L)
\]
as well as
\[
\lim_{t \to \infty} (\sqrt{\Omega_t^U})^{\otimes 2} = (\sqrt{\Omega_\infty^U})^{\otimes 2} \in \bigwedge^{\frac{1}{2}\dim_\RR M} C^\infty(M,T^\ast_{\CC} M) .
\]

\bigskip
\bigskip
\bigskip
{\bf Acknowledgements:} We wish to thank W.D.~Kirwin for sharing the unpublished manuscript \cite{kirwin.wu:manuscript}, where the proof of existence of the Kirwin--Wu polarization was first established, and also where its quantization and partial Bohr--Sommerfeld cycles were first discussed. The broader context of fibering polarizations with its Bohr--Sommerfeld conditions and the central role of the invariant torus action, as well as using imaginary time flow for convergence is new in our treatment. TB, JM and JN were supported by the projects UIDB/04459/2020 and UIDP/04459/2020.

\def\cftil#1{\ifmmode\setbox7\hbox{$\accent"5E#1$}\else
  \setbox7\hbox{\accent"5E#1}\penalty 10000\relax\fi\raise 1\ht7
  \hbox{\lower1.15ex\hbox to 1\wd7{\hss\accent"7E\hss}}\penalty 10000
  \hskip-1\wd7\penalty 10000\box7}
  \def\cftil#1{\ifmmode\setbox7\hbox{$\accent"5E#1$}\else
  \setbox7\hbox{\accent"5E#1}\penalty 10000\relax\fi\raise 1\ht7
  \hbox{\lower1.15ex\hbox to 1\wd7{\hss\accent"7E\hss}}\penalty 10000
  \hskip-1\wd7\penalty 10000\box7}
  \def\cftil#1{\ifmmode\setbox7\hbox{$\accent"5E#1$}\else
  \setbox7\hbox{\accent"5E#1}\penalty 10000\relax\fi\raise 1\ht7
  \hbox{\lower1.15ex\hbox to 1\wd7{\hss\accent"7E\hss}}\penalty 10000
  \hskip-1\wd7\penalty 10000\box7}
  \def\cftil#1{\ifmmode\setbox7\hbox{$\accent"5E#1$}\else
  \setbox7\hbox{\accent"5E#1}\penalty 10000\relax\fi\raise 1\ht7
  \hbox{\lower1.15ex\hbox to 1\wd7{\hss\accent"7E\hss}}\penalty 10000
  \hskip-1\wd7\penalty 10000\box7} \def\cprime{$'$}

%\printbibliography

\end{document}